\documentclass[11pt,a4paper,oneside]{article}
\newcommand{\ignore}[1]{}

\usepackage{a4wide}

\normalbaselines
\usepackage{amsmath}
\usepackage{amssymb}
\usepackage{mathtools}
\usepackage{mathcomp}
\usepackage{textcomp}
\usepackage{amsthm}

\usepackage{cite}
\usepackage{caption}
\usepackage{subcaption}

\usepackage{multicol}
\usepackage{url}
\usepackage{array}
\usepackage{hyperref}
\usepackage{kbordermatrix}

\usepackage{pgfplots}
\pgfplotsset{compat=1.15}
\usepackage{mathrsfs}
\usetikzlibrary{arrows}

\providecommand{\keywords}[1]
{
  \small	
  \textbf{Keywords:} #1
}
\providecommand{\msc}[1]
{
  \small	
  \textbf{MSC classification:} #1
}

\newtheorem{thm}{Theorem}\theoremstyle{plain}
\newtheorem{theorem}[thm]{Theorem}\theoremstyle{plain}
\newtheorem{proposition}[thm]{Proposition}\theoremstyle{plain}
\newtheorem{lemma}[thm]{Lemma}\theoremstyle{plain}

\theoremstyle{plain}

\theoremstyle{plain}
\newtheorem{claim}[thm]{Claim}\theoremstyle{plain}
\theoremstyle{plain}
\theoremstyle{plain}
\theoremstyle{plain}
\theoremstyle{plain}
\newtheorem{problem}{Problem}\theoremstyle{plain}
\theoremstyle{plain}
\theoremstyle{plain}
\newtheorem{conjecture}[thm]{Conjecture}\theoremstyle{plain}
\theoremstyle{plain}

\DeclareMathOperator{\rank}{rank}

\DeclareMathOperator{\Vol}{Vol}
\DeclareMathOperator{\aff}{aff}
\DeclareMathOperator{\conv}{conv}

\DeclareMathOperator{\lk}{lk}
\DeclareMathOperator{\st}{st}

\newcommand{\bd}{{\bf d}}
\newcommand{\be}{{\bf e}}

\newcommand{\cA}{{\mathcal A}}
\newcommand{\cT}{{\mathcal T}}

\newcommand{\cS}{{\mathcal S}}

\newcommand{\R}{{\mathbb R}}
\newcommand{\rat}{{\mathbb Q}}

\newcommand{\sm}{\setminus}
\usepackage{xcolor}
\usepackage{graphicx}

\begin{document}
\title{Volume Rigidity of Simplicial Manifolds}
\author{James Cruickshank\thanks{School of Mathematical and Statistical Sciences, 
University of Galway, Ireland. E-mail: james.cruickshank@universityofgalway.ie}, Bill Jackson\thanks{School of Mathematical Sciences, Queen Mary
University of London, Mile End Road, London E1 4NS, United Kingdom.
E-mail: b.jackson@qmul.ac.uk} and Shin-ichi Tanigawa\thanks{Department of Mathematical Informatics, Graduate School of Information Science and Technology, University of Tokyo, 7-3-1 Hongo, Bunkyo-ku, 113-8656,  Tokyo Japan. Email: {tanigawa@mist.i.u-tokyo.ac.jp}}}

\maketitle

\begin{abstract}
Classical results of Cauchy \cite{C} and Dehn \cite{D} imply that the 1-skeleton of a convex simplicial polyhedron $P$ is rigid, i.e.~every continuous motion of the vertices of $P$  in $\R^3$ which preserves its edge lengths results in a polyhedron which is congruent to $P$. This result was extended to convex simplicial polytopes in $\R^d$ for all $d\geq 3$ by Whiteley \cite {W84}, and to generic realisations of 1-skeletons of simplicial $(d-1)$-manifolds in $\R^{d}$ by Kalai \cite{K} for $d\geq 4$ and Fogelsanger \cite{F}  for $d\geq 3$. We will generalise Kalai's result by showing that, for all $d\geq 4$ and any fixed $1\leq k\leq d-3$, every generic realisation 
of the $k$-skeleton of  a simplicial $(d-1)$-manifold 
in $\R^{d}$ is volume rigid,  
i.e.~every continuous motion of its vertices 
in $\R^d$ which preserves the volumes of 
its  $k$-faces results in a congruent realisation.
In addition, we conjecture that our result remains true for $k=d-2$ and verify this conjecture when $d=4,5,6$.
\end{abstract}


\msc{52C25 (Primary); 05E45, 57Q15 (Secondary)}

\section{Introduction}

We will consider the general problem of deciding when a given realisation $p$ of a hypergraph $H=(V,E)$ in $\R^d$ is {\em volume rigid}: we consider a map $p:V\to \R^d$ and ask if every continuous motion of the points $p(v)$  which preserves the volume of every hyperedge of $H$, results in a realisation which is congruent to $p$. We will primarily be concerned with {\em generic realisations}, i.e.~realisations $p$ such that the multiset of coordinates of the points $p(v)$, $v\in V$, is algebraically independent over $\rat$. In this case we will see that the volume rigidity of $(H,p)$ is determined completely by the underlying hypergraph $H$.

The special case of this problem when $H$ is a graph is a well studied and active area of discrete geometry but the general case of hypergraphs is largely unexplored. 
Classical results of Cauchy \cite{C} and Dehn \cite{D} imply that the 1-skeleton of a convex simplicial polyhedron $P$ is (volume) rigid  in $\R^3$ i.e. every continuous motion of the vertices of $P$ which preserves the edge lengths results in a polyhedron which is congruent to $P$. This result was extended to 1-skeletons of convex simplicial polytopes in $\R^d$ for all $d\geq 3$ by Whiteley \cite{W84}, and to generic realisations of 1-skeletons of simplicial $(d-1)$-manifolds in $\R^{d}$ by Kalai \cite{K} for $d\geq 4$ and Fogelsanger \cite{F}  for $d\geq 3$. We will extend Kalai's result to volume rigidity:

\begin{theorem}\label{thm:d-3}
Let $H$ be the $(k+1)$-uniform hypergraph consisting of the $k$-faces of a connected simplicial $(d-1)$-manifold and $p$ be a  generic realisation of  $H$ in $\R^{d}$ for some $d\geq 4$ and some $1\leq k\leq d-3$. Then $(H,p)$ is volume rigid.
\end{theorem}

We will verify Theorem \ref{thm:d-3} in Section~\ref{sec:jface rigid}, using a relatively short inductive argument  based on Kalai's Theorem and Lemma \ref{lem:simplex} below (which gives the special case of Theorem \ref{thm:d-3} when $H$ is a complete uniform hypergraph).

We conjecture that Theorem \ref{thm:d-3} remains true when $k=d-2$.

\begin{conjecture}\label{con:d-2}
Let $H$ be the $(d-1)$-uniform hypergraph consisting of the $(d-2)$-faces of a connected simplicial $(d-1)$-manifold and $p$ be a  generic realisation of  $H$ in $\R^{d}$ for some $d\geq 3$. Then $(H,p)$ is volume rigid.
\end{conjecture}

Conjecture \ref{con:d-2} holds when $d=3$ by the above mentioned result of Fogelsanger. We will show that our conjecture also holds when either $H$ is a complete $(d-1)$-uniform hypergraph on at least $d+1$ vertices (Theorem \ref{thm:complete}), or $d\in \{4,5,6\}$ 
(Theorem \ref{thm:d=5}). 
These results require substantial effort. In particular we will need a non-trivial extension of Whiteley's Vertex Splitting Lemma to volume rigidity, which we prove in Section \ref{sec:split}. 

We also extend Whiteley's coning lemma to the setting of volume rigidity. While this is not needed to prove our main theorems we include it since it provides a new and non-trivial extension of the classical coning theorem for graph rigidity.

In view of the above results and conjecture, it is natural to ask about the volume rigidity of generic realisations of  the $d$-uniform hypergraph consisting of the $(d-1)$-faces of a connected simplicial $(d-1)$-manifold in $\R^d$. It is not difficult to see that such a hypergraph may not have enough $(d-1)$-faces to be generically volume rigid in $\R^d$. For example, the hypergraph consisting of the 2-faces  of a triangulation of the sphere  
is not generically volume rigid in $\R^3$. This follows since every such triangulation on $n$ vertices has $(2n-4)$ 2-faces and hence every generic realisation in $\R^3$  will have at least $3n-(2n-4)=n+4$ degrees of freedom.  But to be rigid in $\R^3$ it would have to have precisely $6$ degrees of freedom (given by the 6-dimensional space of translations and rotations of $\R^3$).

We can also ask about the volume rigidity of generic realisations of  the $(d+1)$-uniform hypergraph $H=(V,E)$ consisting of the $d$-faces of a connected simplicial $d$-manifold in $\R^{d}$. In this case Borcea and Streinu \cite{BS} observe that such a hypergraph will {\em never} be generically volume rigid in our setting since it will always have at least $d^2+d-1$ degrees of freedom (and would have to have ${d+1\choose2}$ degrees of freedom to be generically rigid in $\R^{d}$ in our setting). On the other hand, 
there are several results which  show that certain families of $(d+1)$-uniform hypergraphs are `volume rigid' in $\R^d$ in the sense that their generic realizations  have precisely $d^2+d-1$ degrees of freedom. Notable examples are the results of Whiteley, see \cite[Section 2.1]{BS}, and
Bulavka, Nevo, and Peled \cite{BNP} which imply 
that the  $3$-uniform hypergraph consisting of the 2-faces of a triangulated surface of low genus is `volume rigid' in $\R^2$. 
The problem of counting the number of distinct realisations of a `volume rigid' $(d+1)$-uniform hypergraph  in $\R^d$
which are equivalent to a given generic  realisation is also explored in \cite{BS,S}. In addition,
Lubetzky and Peled~\cite{LP} use this version of volume rigidity to analyse bootstrap percolation on random hypergraphs.

\section{Volume rigidity of hypergraphs}
In this section we shall give a formal definition of volume rigidity and derive some basic properties.

\subsection{Volume Preserving Affine Maps}
In order to extend the notion of edge-length rigidity to hyperedge-volume rigidity, 
we first need to understand which affine maps preserve volumes of simplices in $\R^d$. {We will use the exterior algebra of $\R^d$ to accomplish this.}
{The material in this subsection is classical and well-known, but we will provide a self-contained exposition since it is fundamental 
{to our understanding} of volume rigidity. }

We use $\R^d$ to denote $d$-dimensional Euclidean space i.e. the $d$-dimensional vector space over $\mathbb{R}$ equipped with the Euclidean inner product $\langle \cdot, \cdot\rangle$.
Let $\be_1,\dots, \be_d$ be the canonical basis of $\mathbb{R}^d$.
For an  integer $k$ with $1\leq k\leq d$, the $k$-th exterior product of vectors $x_1,\dots, x_k\in \mathbb{R}^d$  is denoted by 
$x_1\wedge \dots \wedge x_k$.
The $k$-th exterior power $\bigwedge^k \mathbb{R}^d$ of $\mathbb{R}^d$ is the linear span of the $k$-th exterior products of all possible collections of $k$ vectors in $\mathbb{R}^d$.
An element of $\bigwedge^k \mathbb{R}^d$ is said to be {\em decomposable} if it can be written as $x_1\wedge \dots \wedge x_k$ for some $x_1,\dots, x_k\in \mathbb{R}^d$.

The inner product of $\mathbb{R}^d$ induces an inner product on $\bigwedge^k \mathbb{R}^d$, which is given by
\begin{equation}\label{eq:inner_product_main}
\langle x_1\wedge\dots \wedge x_k, y_1\wedge \dots \wedge y_k \rangle
=\det \begin{bmatrix}
\langle x_1,y_1\rangle  & \langle x_2, y_1\rangle & \ldots  & \langle x_k,y_1\rangle  \\
 \langle x_1, y_2\rangle &  \langle x_2, y_2\rangle  &   & \vdots \\ 
 \vdots &  & \ddots &   \vdots \\
\langle x_1, y_k\rangle & \dots & \dots & \langle x_k, y_k\rangle 
\end{bmatrix}
\end{equation}
and is then extended linearly to all elements of $\bigwedge^k \mathbb{R}^d$. 

We may use this inner product to give a concise expression for the volume of a simplex in $\R^d$. Given  $k \geq 1$ and 
a set of $k+1$ points $P=\{p_1, p_2, \dots, p_{k+1}\}$ in $\mathbb{R}^d$ 
we put
\begin{equation}\label{eq:vol}
 \Vol(P)=\frac{1}{k!}\|(p_2-p_1)\wedge (p_3-p_1)\wedge \dots \wedge (p_{k+1}-p_1)\|,     
 \end{equation}
where $\| \cdot \|$ denotes the norm defined by the inner product $\langle \cdot, \cdot\rangle$ on $\bigwedge^k \R^d$.
When $k\leq d$, $\Vol(P)$ is the square of the volume of the simplex defined by $P$ in any $k$-dimensional affine subspace of $\R^d$ which contains $P$. Note that $x_1\wedge\dots \wedge x_k={\bf 0}$ whenever  the multiset $\{x_1,x_2,\ldots,x_k\}$ is linearly dependent and hence $\Vol(P)=0$ whenever either  $k>d$, or, $k\leq d$ and the points in $P$ are not affinely independent.

We next turn to the problem of determining which affine maps preserve $\Vol(P)$ for all sets $P$ of $k+1$ points in $\R^d$. We may assume that $k\leq d$ since $\Vol(P)=0$ whenever $k>d$.  In this case we may use Pl{\"u}cker coordinates (defined below) to express the inner product 
as the ordinary scalar product in a Euclidean space.

Consider the ${d\choose k}$-dimensional Euclidean space $\mathbb{R}^{d\choose k}$ whose coordinates are indexed by the $k$-subsets of $\{1,\dots, d\}$.
Let $[d]=\{1,\dots, d\}$ and denote by ${[d] \choose k}$  the set of all $k$-subsets of $[d]$.
For clarity of the subsequent discussion, we will always assume that 
the coordinates of $\mathbb{R}^{d\choose k}$ are ordered in the lexicographical ordering of the $k$-sets $I\in {[d]\choose k}$.
For example, if $d=4$ and $k=2$, then the coordinates are indexed in the order $12, 13, 14, 23, 24, 34$.

The  {\em Pl{\"u}cker coordinate} of an 
element 
of $\bigwedge^k \mathbb{R}^d$ is the 
vector in $\mathbb{R}^{d\choose k}$
given by taking the
coordinate vector of 
this element
with respect to the canonical basis
$\{\be_{i_1}\wedge \dots \wedge \be_{i_k}: I=\{i_1,\dots, i_k\}\in {[d]\choose k}\}$ of $\bigwedge^k \mathbb{R}^d$. 
Given a decomposable 
element 
$x_1\wedge \dots \wedge x_k\in 
\bigwedge^k \mathbb{R}^d$,
let $X$ be the $d\times k$ matrix whose $i$-th column is $x_i$,
and let  $X_I$ be the square submatrix of $X$ obtained by picking the rows indexed by $i_1,\dots, i_k$ for $I=\{i_1,\dots, i_k\}\in {[d]\choose k}$.
Then, the $I$-th Pl{\"u}cker coordinate of $x_1\wedge \dots \wedge x_k$ is $\det X_I$. 
The Cauchy-Binet formula now implies that, for $x_1\wedge \dots \wedge x_k$
and $y_1\wedge \dots \wedge y_k$ in $\bigwedge^k \mathbb{R}^d$, 
\begin{align*}
\langle x_1\wedge\dots \wedge x_k, y_1\wedge \dots \wedge y_k \rangle
=\det(X^{\top} Y)
=\sum_{I\in \binom{[d]}{k}} \det X_I \det Y_I,
\end{align*}
and the latter term is indeed the scalar product of the Pl{\"u}cker coordinates of $x_1\wedge \dots \wedge x_k$
and $y_1\wedge \dots \wedge y_k$ in $\mathbb{R}^{d\choose k}$. 
This correspondence can be extended to inner products of all pairs of elements of  $\bigwedge^k \mathbb{R}^d$ by linearity.

Given an $m\times n$ matrix $A$ and an integer $k$ with $1\leq k\leq \min\{m,n\}$, the {\em $k$-th compound matrix $A^{(k)}$} is the ${{m}\choose{k}}\times {{n}\choose{k}}$ matrix whose rows and columns are indexed by the lexicographic orderings of 
${[m] \choose k}$ and ${[n]\choose k}$, respectively, and whose entries are the corresponding $k\times k$ minors of $A$. 
The compound matrix encodes the Pl{\"u}cker coordinates of all possible
exterior products of $k$ vectors chosen from the $n$ column vectors of $A$.
Indeed, for $I=\{i_1, \dots, i_k\}$ with $i_1 < \dots < i_k$, 
the $I$-th column of $A^{(k)}$ is $a_{i_1}\wedge \dots \wedge a_{i_k}$,
where $a_{i_j}$ denotes the $i_j$-th column of $A$.
From this, one can easily check 
that
\[
Ax_1\wedge \dots \wedge Ax_k = A^{(k)} (x_1\wedge \dots \wedge x_k)
\]
for any $x_1,\dots, x_k\in \mathbb{R}^n$.

We will need the following well known properties of compound matrices, see for example \cite{H}.
\begin{lemma}\label{lem:compound}
Let $A$ be an $n\times n$ matrix and $k$ be a integer with $1\leq k\leq n$. Then:

(a) 
$\det A^{(k)}=(\det A)^{{n-1\choose k-1}}$.

(b) When $ k \leq n-1$, $A$ is orthogonal if and only if $A^{(k)}$ is orthogonal.
\end{lemma}

We can now  determine which  affine maps preserve $\Vol(P)$ for all sets $P$ of $k$ points in $\R^d$ when $k\leq d$.
We note, for a reader who is unfamiliar with volumes of simplices,
that this problem is not completely trivial since 
the situation for 1-simplices is different to that for $d$-simplices when $d\geq 2$.
Indeed, if 
$P=\{p_1,\dots, p_{d+1}\}$,
then $d!{\rm Vol}(P)$ is the absolute value of the determinant of the $d\times d$ matrix with columns $p_2-p_1, \dots, p_{d+1}-p_1$.
Hence, a map $p\mapsto Ap+t$ for some $A\in \mathbb{R}^{d\times d}$ and $t\in \mathbb{R}^d$ preserves the volume of each $d$-simplex if and only if $|\det A|=1$.
This in particular implies that the space of such affine maps is
$(d^2+d-1)$-dimensional, see~\cite{BNP,LP,BS} for more details. On the other hand, it is well known that the map $p\mapsto Ap+t$ 
preserves the volume of each $1$-simplex in $\R^d$, i.e. the length of all line segments in $\R^d$, if and only if $A$ is orthogonal and hence the dimension of the space of all affine maps which preserve the volume of all 1-simplices in $\R^d$ is ${d+1}\choose 2$.

Combined with the expression for the volume of a set of $k+1$ points in $\R^d$ given by (\ref{eq:vol}), our next proposition implies that 
the case $k=1$ extends to all $1\leq k\leq d-1$, and it is the case $k=d$ that is exceptional.

\begin{proposition}\label{prop:isometries_main}
Let $A$ be a real $d\times d$ matrix  and $k$ be an integer with $1\leq k< d$.
Then, 
\begin{equation}\label{eq:prop_isometries_main}
\|Ax_1\wedge Ax_2\wedge \dots \wedge Ax_k\|=\|x_1\wedge x_2\wedge \dots \wedge x_k\| 
\end{equation}
for all $x_1, x_2, \dots, x_k\in \mathbb{R}^d$
if and only if 
$A$ is orthogonal.
\end{proposition}
\begin{proof}
Since $Ax_1\wedge  \dots \wedge Ax_k=A^{(k)} (x_1\wedge \dots \wedge x_k)$,
(\ref{eq:prop_isometries_main}) holds if and only if 
$\|A^{(k)}(x_1\wedge x_2\wedge \dots \wedge x_k)\|=
\|x_1\wedge x_2\wedge \dots \wedge x_k\|$
for all $x_1, x_2, \dots, x_k\in \mathbb{R}^d$.
Since the set of decomposable elements spans $\bigwedge^k\mathbb{R}^d$
and $\|\cdot \|$ is the Euclidean norm of $\bigwedge^k\mathbb{R}^d$ with respect to the canonical basis,
this condition is equivalent to the orthogonality of $A^{(k)}$.
This in turn is equivalent  to the orthogonality of $A$ by Lemma~\ref{lem:compound}.
\end{proof}
Note that the  argument in the proof of Proposition~\ref{prop:isometries_main}
can be easily adapted for the case when $k=d$.
In this case, we have $\bigwedge^d\mathbb{R}^d\simeq \mathbb{R}$,
and hence   
(\ref{eq:prop_isometries_main}) holds for $k=d$ if and only if 
$|\det A^{(d)}|=1$. This condition is equivalent to $|\det A|=1$ by Lemma~\ref{lem:compound}.

In view of applications to volume rigidity (discussed below), it would be of interest to  extend Proposition~\ref{prop:isometries_main} to the situation when \eqref{eq:prop_isometries_main} holds for all $k$-subsets of a finite set of points. 
To the best of our knowledge, the following problem is still open.
\begin{problem}
Let $d,k$ be positive integers with $1\leq k<d$.
Characterise when a finite set  of points $P$ in $\mathbb{R}^d$ has the following property: for every $d\times d$ matrix $A$, 
$\|Ax_1\wedge A x_2\wedge \dots \wedge A x_k\|=
\|x_1\wedge x_2\wedge \dots \wedge x_k\|$
for all $x_1, x_2,\dots, x_k\in P$ if and only if 
$A$ is orthogonal.
\end{problem}

\subsection{Volume rigidity}\label{sec:volrig}
In view of Equation (\ref{eq:vol}) and Proposition~\ref{prop:isometries_main}, we may define volume rigidity 
as follows.
Given two realisations $p, q$ of a hypergraph $H=(V,E)$ in $\mathbb{R}^d$,
we say that $(H,p)$ is {\em volume equivalent} to $(H,q)$ if 
${\rm Vol}(p(\Delta))={\rm Vol}(q(\Delta))$ holds for all $\Delta\in E$ and that
$p$ is {\em congruent} to $q$ if 
there is an orthogonal matrix $A$ and a vector $t\in \mathbb{R}^d$ such that
$q(v)=A p(v)+t$ for all $v\in V$. Then
$(H,p)$ is {\em volume rigid} if 
there is a neighborhood $N_p$ of $p$ in $\R^{d|V|}$ 
such that, for all $q\in N_p$ such that $(H,q)$ is volume equivalent to $(H,p)$, $q$ is congruent to $p$.
When $H$ is a graph, this definition coincides with the standard definition of the  rigidity of the bar-joint framework $(H,p)$.

A fundamental property of bar-joint rigidity is that a realization of a complete graph is rigid as long as the points of the realization are in general position.
The corresponding property does not always hold for volume rigidity.
For example, if $H$ is the complete $d$-uniform hypergraph on $d+1$ vertices,
then the number of hyperedges of $H$ is too small 
for a generic realization of $H$ in $\R^d$ to be volume rigid when $d\geq 3$.

Characterising which realisations of the complete $(k+1)$-uniform\footnote{We use $k+1$ here instead of $k$ since the 'simplicial dimension' of a $(k+1)$-set is $k$} hypergraph $K_n^{k+1}$ are volume rigid appears to be a challenging open problem. Our next result comes close to solving this problem for the special case of generic realisations.

\begin{theorem}\label{thm:complete}
Suppose $n,d,k$ are integers with $1\leq k\leq \min \{d-2,n-3\}$.
Let $K^{k+1}_n$ be the complete $(k+1)$-uniform hypergraph with $n$ vertices
and $p:V(K^{k+1}_n)\rightarrow \mathbb{R}^d$ be  generic.
Then $(K^{k+1}_n,p)$ is volume rigid.
\end{theorem}

The special case of Theorem \ref{thm:complete} when 
$k\leq n-3 \leq d-2$ will be proved in the next subsection. The general case will be derived from this special case in Section \ref{sec:comb}  using a combinatorial lemma given  in the same section. 

Assuming that $n\geq d+1$, so that generic realisations of $K_n^{k+1}$ affinely span $\R^d$,
and that $k\leq d$, we are left to consider the cases when $k\in\{d-1,d\}$. The complete hypergraph   $K_n^{d+1}$ is not generically volume rigid in $\R^d$ for all $n$ by the above mentioned result of 
Borcea and Streinu~\cite{BS} or 
Lubetzky and Peled \cite{LP}. The complete hypergraph   $K_{n}^{d}$ does not have enough hyperedges to be  generically volume rigid in $\R^d$ when $n=d+1$,
but is generically volume rigid  when $n\geq d+2$ by Theorem \ref{thm:LNPR} below.

\subsection{Infinitesimal volume rigidity}

The {\em $d$-dimensional volume rigidity map} of a hypergraph $H=(V,E)$
is the function $f:\R^{d|V|}\to \R^{|E|}$ given by
$$f(p)=(f_p(\Delta_1),f_p(\Delta_2),\ldots,f_p(\Delta_m))$$
where  $E=\{\Delta_1,\Delta_2,\ldots,\Delta_m\}$  and $f_p(\Delta_h)=\Vol( p(\Delta_h))^2$ is as defined by (\ref{eq:vol}).  

We will define the {\em volume rigidity matrix} $R(H,p)$ of a realisation $p$ of $H$ in $\R^d$ to be a matrix obtained from the Jacobian of $f$ evaluated at $p$ by suitable scalar multipications of its rows, 
and will associate the rows of $R(H,p)$ with the hyperedges of $H$
and sets of $d$ consecutive  columns with its vertices.

Let $\Delta=\{v_0,v_1,\ldots,v_k\}$ be a hyperedge of  $H$.
For each $0\leq i\leq k$, let $p_i = p(v_i)$ and let $p_i^\Delta$ be the projection of $p_i$ onto the affine subspace 
of $\R^d$ spanned by $p(\Delta-v_i)$.
Then  (\ref{eq:inner_product_main}) and (\ref{eq:vol}) give 
$$\Vol( p(\Delta))^2=\begin{cases}
\frac{1}{k^2}\|p_i-p_i^\Delta\|^2 \Vol (p(\Delta-v_i))^2 & (k\geq 2) \\
\|p_1-p_0\|^2 & (k=1) 
\end{cases}
$$
Hence, if $k=1$, then  we can take the entry in row $\Delta$ and columns $v_i$ of $R(H,p)$ to be  
$(p_i-p_j)$
when $\{v_i,v_j\}=\Delta$ and to be $\bf 0$ when $v_i\not\in \Delta$.

It remains to consider the case when $|\Delta|\geq 3$.
Since 
$p_i^\Delta=P_{\Delta-v_i}(p_i-x)+x$ for any $x\in \aff (p(\Delta))$, where $P_{\Delta-v_i}$ is the matrix which projects $\R^d$ onto its $(k-1)$-dimensional linear subspace ${\rm aff}(p(\Delta-v_i))-x$, this gives
\begin{align*}
k^2 \frac{\partial \Vol( p(\Delta))^2}{\partial p_i}
&=2(I-P_{\Delta-v_i})(p_i-p_i^\Delta) \Vol (p(\Delta-v_i))^2\\
&=2(p_i-p_i^\Delta) \Vol( p(\Delta-v_i))^2.
\end{align*}
where the second equality follows from the fact that 
$p_i-p_i^\Delta$ is orthogonal to ${\rm aff}(p(\Delta-v_i))$.
Hence, when $|\Delta|\geq 3$, we can take the entry in row $\Delta$ and columns $v_i$ of $R(H,p)$ to be  
\begin{equation}\label{eq:rigmatrix}
(p_i-p_i^\Delta) \Vol( p(\Delta-v_i))^2
\end{equation}
when $v_i\in \Delta$ and to be $\bf 0$ when $v_i\not\in \Delta$.

\medskip

An {\em infinitesimal  motion} of $(H,p)$ is a map $\dot{p}:V\rightarrow \mathbb{R}^d$
which belongs to the right kernel of $R(H,p)$ (when regarded as a vector in $\R^{d|V|}$).
Since every Euclidean isometry preserves the volume of each simplex in $\R^d$, 
the map $v\mapsto Sp(v)+t$ for all $v\in V$  will be an infinitesimal motion of $(H,p)$ whenever $S$ is a  skew-symmetric $d\times d$ matrix and $t\in \mathbb{R}^d$. We will refer to such motions as  {\em trivial} infinitesimal motions of $(H,p)$.  

We say that
$(H,p)$ is {\em infinitesimally volume rigid} if
every infinitesimal motion of $(H,p)$ is trivial.
A well-known fact from rigidity theory is that the space of trivial infinitesimal motions has dimension
${d+1\choose 2}$ if $p(V)$ affinely spans $\R^d$.
More generally, the dimension of 
the space of trivial infinitesimal motions is 
${{d+1}\choose 2}-{d-d_p\choose 2}$ when the dimension of ${\aff} (p(V))$ is $d_p$.
Hence, $(H,p)$ is infinitesimally volume rigid if and only if
$\dim \ker R(H,p)={d+1\choose 2}-{d-d_p\choose 2}$.
In particular, when $p(V)$ affinely spans $\mathbb{R}^d$, $(H,p)$ is infinitesimally volume rigid if and only if
\[
\rank R(H,p)=d|V|-\binom{d+1}{2}.
\]

As noted in the last subsection, the case when $H$ is $(d+1)$-uniform is exceptional and there exists a stronger lower bound on the dimension of the space of infinitesimal  motions of  $(H,p)$ in this case.
This led  
previous authors~\cite{BS,LP,BNP}
to use a different definition for the infinitesimal volume rigidity of realisations of $(d+1)$-uniform hypergraphs in $\R^d$. We refer the reader to their papers for more details.

Note that, since the volume rigidity map $f$ is a polynomial map with rational coefficients,  $f(p)$ will be a  regular value of  $f$ whenever $p$ is generic, and hence $f^{-1}(f(p))$ will be a manifold of dimension $d|V|-\rank R(H,p)$. Since the submanifold generated by the trivial motions of $(H,p)$ has dimension $d|V|-\binom{d+1}{2}+\binom{d-d_p}{2}$ by \cite{AR}, this gives:
\begin{proposition}\label{prop:equivalence}
Let $H=(V,E)$ be a hypergraph 
and $p:V\rightarrow \mathbb{R}^d$ be generic.
Then $(H,p)$ is volume rigid if and only if it is infinitesimally volume rigid.
\end{proposition}

Proposition \ref{prop:equivalence} allows us to define a hypergraph $H$ 
as being {\em volume rigid in $\R^d$} if some, or equivalently every, generic realisation of $H$ in $\R^d$ is infinitesimally volume rigid.

We close this section by using Proposition \ref{prop:equivalence} to verify the special case  of Theorem \ref{thm:complete} when $n \leq d+1$. Our proof also uses a result of Lov\`asz \cite[Theorem 13]{L} which determines the spectrum of the Kneser graph $K(n,k)$.
(This is the graph with vertex set ${{[n]}\choose k}$ in which two vertices $X,Y$ are adjacent whenever $X\cap Y=\emptyset$.)

\begin{lemma} \label{lem:simplex} The complete $(k+1)$-uniform hypergraph $K^{k+1}_{n}$  is   volume rigid in $\R^{d}$ for all $1 \leq k \leq n-3\leq d-2$.
\end{lemma}
\begin{proof} 
We will proceed by induction on $n-k$. We first consider the base case when $n-k=3$.
Let $K_{k+3}^{k+1}=H=(V,E)$. 
Suppose that $d \geq k+3 = n$ and that $\dot{p}: V \to \R^d$ is an  infinitesimal motion of a framework $(H,p)$ for some generic realisation $p$ of $H$ in $\R^d$. Since $|V|= k+3$ it follows that $p(V)$ is contained in some $(k+2)$-dimensional affine subspace $Y$ of $\R^d$. Choose $y_0 \in Y$ and put $L={\rm span}\{y-y_0:y\in Y\}$.  Then $L$ is a $(k+2)$-dimensional linear subspace of $\R^d$  and $p':v\mapsto p(v)-y_0$ is a realisation of $H$ in $L$. Since $|V| = k+3\leq d$, there is a trivial infinitesimal motion $\dot{q}:V \to \R^d$ such that $\dot{p}(u)+\dot{q}(u) \in L$ for all $u\in V$. Now $\dot{p} + \dot{q}$ is trivial if and only if $\dot{p}$ is trivial. It follows that $(H,p)$ is volume rigid in $\R^d$ if and only if {$(H,p')$} is volume rigid in the $(k+2)$-dimensional space $L$. 
Therefore we can assume that $d=k+2 = n-1$.

Since the entries in the rigidity matrix $R(H,p)$ are rational functions of the coordinates of any realisation $p:V\to \R^d$, it will suffice to show that some realisation $p$ of $H$ in $\R^{d}$ is infinitesimally volume rigid. Let $V=\{v_1,v_2,\ldots,v_{d+1}\}$. Consider the realisation $p$ of $H$ in $\R^{d+1}$ given by $p(v_i)=\be_i$, where $\be_1,\be_2,\ldots,\be_{d+1}$ is the standard basis for $\R^{d+1}$. Since the vertices of $(H,p)$ all lie in a hyperplane of $\R^{d+1}$ and the rank of $R(H,p)$ is invariant under translations and rotations  in $\R^{d+1}$, it will suffice to show that 
\begin{equation}\label{eq:e1}
\rank R(H,p)\geq |V(H)|d-\binom{d+1}{2}=\binom{d+1}{2}.
\end{equation}

By symmetry, each hyperedge of $H$ will have the same volume $W$ in $(H,p)$. Let $M_1$ be the matrix obtained by dividing each row of $R(H,p)$ by $W$. Then, for each $\Delta\in E(H)$ and $v_i\in V$, the vector in row $\Delta$ and columns $v_i$  of $M_1$ is $\be_i-\frac{1}{k}\sum_{v_j\in \Delta-v_i}\be_j$ if $v_i\in \Delta$, and $\bf 0$ otherwise. 
 
For each $1\leq i\leq d+1$, label the columns in the $v_i$-block of columns of  $M_1$ consecutively with the ordered pairs $(h,i)$, $1\leq h\leq d+1$. Let  $M_2$ be the matrix obtained by multiplying each row of $M_1$ by $-k$ and then dividing the $(i,i)$-th column of  $M_1$  by  $-k$ for each $1\leq i\leq d+1$. Then, for each $\Delta\in E(H)$ and $v_i\in V(H)$, the vector in row $\Delta$ and columns $v_i$  of $M_2$ is  $\sum_{v_j\in \Delta}\be_j$ if  $v_i\in \Delta$, and $\bf 0$ otherwise.  
 
Finally, let $M_3$ be the matrix obtained from $M_2$ by deleting the  columns  labelled $(h,i)$ for all $1\leq i\leq h\leq k+3$. Then $M_3$ is a $\binom{d+1}{2}\times \binom{d+1}{2}$-matrix in which the entry in row $\Delta$ and column $(h,i)$ is $1$ if $\{v_h,v_i\}\subseteq \Delta$ and $h<i$, and $0$ otherwise. We can now see that $M_3$ is the adjacency matrix of the Kneser graph $K(d+1,2)$ by changing the label on row $\Delta$ to $(h,i)$ when $V(H)\sm \Delta=\{v_h,v_i\}$ and $h<i$, and then apply \cite[Theorem 13]{L} to deduce that $M_3$ is non-singular.\footnote{More precisely the Claim in the proof of \cite[Theorem 13]{L} tells us that zero is not an  eigenvalue of $M_3$.} This gives $\rank R(H,p)\geq \rank M_3=\binom{d+1}{2}$ and hence (\ref{eq:e1}) holds. This completes the proof when $k=n-3$.

We can now proceed to the induction step. Suppose that $n-k\geq 4$, $K = (V,E) = K^{k+1}_{n} $ and let $\dot{p}$ be an infinitesimal motion of $(K,p)$ for some generic $p: V \to \R^d$. Let $X$ be a $(k+2)$-subset of $V$ and choose a $(n-1)$-subset $V'\subset V$ such that $X \subset V'$. Let $K'$ be the copy of $K^{k+1}_{n-1}$ with vertex set $V'$. By induction $(K',p|_{V'})$ is volume rigid in $\R^d$. Therefore $\dot{p}$ restricts to a trivial infinitesimal motion of $(K',p|_{V'})$. In particular, $\dot{p}$ infinitesimally preserves the volume of the $(k+1)$-simplex $p(X)$. Since $X$ was an arbitrary $(k+2)$-subset of $V$, we can conclude that $\dot{p}$ is an infinitesimal motion of $(K'',p)$ where $K''$ is the copy of $K^{k+2}_{n}$ with vertex set $V$. By induction $(K'',p)$ is rigid in $\R^d$.
Hence $\dot p$ is a trivial infinitesimal motion and $(K,p)$ is infinitesimally volume rigid, as required.

\end{proof}

\section{Three rigidity lemmas}\label{sec:comb}
We will extend three fundamental lemmas for graph rigidity to hypergraph volume rigidity and apply the first lemma to complete the proof of Theorem \ref{thm:complete}. Versions of the first two lemmas (for the special case of $(d+1)$-uniform hypergraphs in $\R^d$) were previously 
given in \cite{BS,BNP}.

\subsection{A gluing lemma and the proof of Theorem \ref{thm:complete}}

The operation of gluing together two rigid graphs to construct a larger rigid graph is a standard technique in graph rigidity theory, see \cite{Wlong}. It is straightforward to extend this to hypergraph volume rigidity.

\begin{lemma}\label{lem:glue_k} Let $H_1,H_2$ be two hypergraphs which are volume rigid in $\R^{d}$ and satisfy $|V(H_1)\cap V(H_2)|\geq d$.
Then $H=H_1\cup H_2$ is  volume rigid in $\R^{d}$. 
\end{lemma}
\begin{proof} 
Choose a generic realisation $p$ of $H$ in $\R^{d}$ and 
let $\dot{p}$ be an infinitesimal motion of  $(H,p)$. Then $\dot{p}_i=\dot{p}|_{V(H_i)}$ is an infinitesimal  motion of $(H_i,p|_{H_i})$. The hypothesis that $H_i$ is volume rigid in $\R^{d}$ now implies that $\dot{p}_i$ is trivial infinitesimal motion of $(H_i,p|_{H_i})$. So, for $i=1,2$, $\dot{p_i}(v) = A_iv+t_i$ for some skew symmetric $d\times d$ matrix $A_i$ and $t_i \in \R^d$ and all $v \in V(H_i)$.  Since $\dot{p}_1$ and $\dot{p}_2$ agree on at least $d$ affinely independent points in $\R^{d}$, it follows that $A_1 = A_2$ and $t_1 = t_2$ and so $p$ is a trivial infinitesimal motion of $(H,p)$. Hence $(H,p)$ is volume rigid.
\end{proof}

\begin{proof}[Proof of Theorem \ref{thm:complete}]
By Lemma \ref{lem:simplex} we can assume that $n\geq d+2$. We proceed by induction on $n$. Choose $ H,K\subset K^{k+1}_n$, where $H$ and $K$ are both copies of $K^{k+1}_{n-1}$ and $|V(H) \cap V(K)| = n-2\geq d$. By induction both $H$ and $K$ are volume rigid in $\R^d$ and so Lemma \ref{lem:glue_k} implies that $H\cup K$ is a volume rigid spanning subgraph of $K^{k+1}_n$.
\end{proof}

\subsection{A vertex splitting lemma}
\label{sec:split}
Given a hypergraph $H$ and two vertices $u,v$ of $H$, we can transform $H$ into a new hypergraph $H/uv$ by deleting every hyperedge which contains  $\{u,v\}$ and replacing every hyperedge $\Delta$ which contains $u$ and not $v$ by $\Delta-u+v$ (when $\Delta-u+v\not\in \cS$). We will say that 
$H/uv$ is obtained from $H$ by {\em contracting $u$ onto $v$} and that  
$H$ can be obtained from $H/uv$ by {\em splitting $u$ from $v$}. Note that the contraction operation is well defined but we can obtain different hypergraphs by splitting $u$ from $v$, depending on our choice for the sets of hyperedges incident to $u$ and $v$. The graphical version of these operations was first used to study the rigidity of bar-joint frameworks by Whiteley \cite{Wvsplit}, who gave sufficient conditions for vertex splitting to preserve rigidity. His vertex splitting lemmas \cite[Proposition 10]{Wvsplit} and \cite[Theorem 9.3.7]{Wlong} have many applications in the study of bar-joint rigidity. In particular, \cite[Proposition 10]{Wvsplit} is an essential component of Fogelsanger's proof \cite{F}  that the 1-skeleton of a connected simplicial $(d-1)$-manifold is generically rigid in $\R^d$ when $d \geq 3$. 

Borcea and Streinu \cite[Proposition 3]{BS} and Bulavka et al \cite[Lemma 3.1]{BNP} give a sufficient condition for vertex splitting to preserve the property that the volume rigidity matrix of a generic realisation of a $(d+1)$-uniform hypergraph has maximum possible rank in $\R^d$.
We will adapt Whiteley's original proof technique to obtain a similar result for an arbitrary hypergraph $H=(V,E)$ and $u,v\in V$. It gives a common generalisation of both of Whiteley's original vertex splitting lemmas to hypergraphs. Our sufficient condition for vertex splitting to preserve volume rigidity is in terms of the rank of a matrix $A_{uv}(H,p,\bd)$, defined for a given realisation $p$ of the hypergraph $H/uv$ obtained by contracting $u$ onto $v$, and a given direction $\bd\in \R^d$.   

Let $E_u$ and $E_v$ be the sets of hyperedges of $H$ which contain $u$ and $v$, respectively. Put $E_{uv}=E_u\cap E_v$, 
$E_v^u=\{\Delta\in E_v\sm E_u:\Delta-v+u\in E_u\}$, 
and $E_u^v=\{\Delta\in E_u\sm E_v:\Delta-u+v\in E_v\}$. Recall that for each $x\in \Delta\in E(H/uv)$, $p(x)^\Delta$ denotes the projection of $p(x)$ onto the affine subspace of $\R^d$ spanned by $p(\Delta-x)$. For each $\Delta\in E_{uv}$ let $\bd_\Delta$ denote the projection of $\bd$ onto the orthogonal complement of the linear subspace of $\R^d$ spanned by $\{p(x)-p(v):x\in \Delta-u\}$ (so, in particular, $\bd_\Delta=\bd$ when $|\Delta|=2$). Finally,
let $A_{uv}(H,p,\bd)$ be the matrix with rows given by the vectors in 
$$\{\bd_\Delta:\Delta\in E_{uv}\}\cup \{p(v)-p(v)^\Delta:\Delta\in E_v^u\}.$$
Note that the entries of $A_{uv}(H,p,\mathbf{d)}$ are rational functions (with coefficients in $\mathbb Q$) of the coordinates of $p$ and $\mathbf{d}$. This follows from the following lemma.
\begin{lemma}
    \label{lem:rational}
    Let $w_1,\dots,w_l \in \R^d$ be affinely independent points and let $\pi:\R^d \to \aff(w_1,\dots,w_l)$ be the orthogonal projection. Then, for each $u\in \R^d$, the coordinates of $\pi(u)$ are expressible as rational functions (with coefficients in $\mathbb Q$) of the coordinates of $w_1,\dots,w_l,u$.
\end{lemma}
\begin{proof}
    Write $\pi(u) = \sum_{i=1}^l c_iw_i$ for some $c_1,\dots, c_l\in \mathbb{R}$. Then $(c_1,\dots,c_l)$ is the unique solution to the following system of equations, that is linear in variables $x_1,\dots,x_l$. 
    $$ \begin{array}{rlcl}
    \sum_{i=1}^l x_i & = & 1 &\\
    \langle u -\sum_{i=1}^lx_iw_i, w_1 -w_s \rangle & = & 0 & \text{ for }s = 2,\dots,l 
    \end{array}
    $$
    The lemma now follows from Cramer's Rule.
\end{proof}

    \begin{lemma}\label{lem:vsplit_kface} Let $H=(V,E)$ be a hypergraph,
    $u,v$ be distinct vertices of $H$ and $q$ be an infinitesimally volume rigid realisation of $H/uv$ in $\R^d$. Suppose  that $\rank A_{uv}(H,q,\bd)=d$ for some $\bd\in \R^d$. Then $H$ is  volume rigid in $\R^d$. 
\end{lemma}
\begin{proof} 
We construct a realisation $p$ of $H$ in $\R^{d}$ by putting $p(w)=q(w)$ for all $w\in V-u$ and $p(u)=q(v)$. 
Recall that,  for any  $\Delta\in E$ and  $w\in V$, the entries in row $\Delta$ and columns $w$ of the rigidity matrix $R(H,p)$ are given by 
 $$g(w,\Delta,p):=
 \begin{cases}
 (p(w)-p(w)^\Delta) \Vol( p(\Delta-w))^2 & \mbox{when $w\in \Delta$ and $|\Delta|\geq 3$}\\  
 p(w)-p(w)^\Delta & \mbox{when $w\in \Delta$ and  $|\Delta|=2$},
 \end{cases}
 $$
and are equal to zero when $w\not\in \Delta$.
Since $p(u)=p(v)$, 
 for each $\Delta_i\in E_{uv}$,
  $p(w)-p(w)^{\Delta_i}={\bf 0}$ for all $w\in \{u,v\}$ and $\Vol (p(\Delta_i-w))=0$ for all $w\in \Delta_i\sm \{u,v\}$.
  Hence the row of $R(H,p)$ indexed by $\Delta_i\in E_{uv}$ is zero.
Thus   $R(H,p)$ has the following form:
\renewcommand{\arraystretch}{1.3}
\[
\kbordermatrix{
 & u & v &  \\
 & \vdots & \vdots & \vdots \\
\Delta_i\in E_{uv} & 0 & 0 & 0 \\
 & \vdots & \vdots & \vdots \\\cline{2-4}
  & \vdots & \vdots & \vdots \\
\Delta'_i\in E_u^v & g(u,\Delta'_i,p) & 0 &\ast \\
 & \vdots & \vdots & \vdots \\\cline{2-4}
 & \vdots & \vdots & \vdots \\
\Delta''_i\in E_v^u & 0 & g(v,\Delta''_i,p) & \ast \\
 & \vdots & \vdots & \vdots \\\cline{2-4}
 & \vdots & \vdots & \vdots \\
\Delta'''_i\in E_u\setminus (E_{uv}\cup E_u^v) & g(u,\Delta'''_i,p) & 0 & \ast \\
 & \vdots & \vdots & \vdots \\\cline{2-4}
 & \vdots & \vdots & \vdots \\
\Delta''''_i\in E_v\setminus (E_{uv}\cup E_v^u) & 0 & g(v,\Delta''''_i,p) & \ast \\
 & \vdots & \vdots & \vdots \\\cline{2-4} 
E(H-u-v) & 0 & 0 & R(H-u-v,p|_{V-u-v})\\
},\]
where 
$H-u-v$ is the hypergraph obtained from $H$ by deleting $u,v$ and  all hyperedges containing $u$ or $v$.
By the definition of $E_u^v$ and $E_v^u$, we may denote 
$E_u^v=\{\Delta'_1,\dots, \Delta'_s\}$ and 
 $E_v^u=\{\Delta''_1,\dots, \Delta''_s\}$ 
and suppose $\Delta''_i = \Delta'_i-u+v$ for each $i$ with $1\leq i\leq s$. 
Then, since $p(u)=p(v)$, we  have
\begin{equation}\label{eq:vsplitting1}
g(u,\Delta_i',p)=g(v,\Delta_{i}'',p) \text{ and } g(w,\Delta_i',p)=g(w,\Delta_{i}'',p) 
\text{ for all  $w\in \Delta'_i-u$ and all $1\leq i\leq s$}.
\end{equation}

Construct a new matrix $R^*$ from $R(H,p)$ by replacing  the zero row indexed by $\Delta_i\in E_{uv}$ by a row with $\bd_{\Delta_i}$ in the columns indexed by $u$,  $-\bd_{\Delta_i}$ in the columns indexed by $v$ and zeros elsewhere, 
i.e., the top row-block of $E_{uv}$ in $R(H,p)$ is replaced with  
\[
 \kbordermatrix{
  & u & v &  \\
  \vdots & \vdots & \vdots &\vdots\\
  \Delta_i\in E_{uv}  & \bd_{\Delta_{i}} & -\bd_{\Delta_{i}} & \hspace{3em} {\bf 0} \hspace{3em}\\
  \vdots & \vdots & \vdots &\vdots
   }.
   \]
In $R^*$, we can now add each column labelled by $u$ to the corresponding column labelled by $v$, then subtract the row labelled by $\Delta_i''$ from the row labelled by $\Delta_i'$ for all  $\Delta_i'\in E_u^v$, to deduce that the rank of  $R^*$ is 
equal to the rank of the following matrix:
\[
\kbordermatrix{
 & u & v &  \\
 & \vdots & \vdots & \vdots \\
\Delta_i\in E_{uv} & d_{\Delta_i} & 0 & 0 \\
 & \vdots & \vdots & \vdots \\\cline{2-4}
  & \vdots & \vdots & \vdots \\
\Delta'_i\in E_u^v & g(u,\Delta'_i,p) & 0 & 0 \\
 & \vdots & \vdots & \vdots \\\cline{2-4}
 & \vdots & \vdots & \vdots \\
\Delta''_i\in E_v^u & 0 & g(v,\Delta''_i,p) & \ast \\
 & \vdots & \vdots & \vdots \\\cline{2-4}
 & \vdots & \vdots & \vdots \\
\Delta'''_i\in E_u\setminus (E_{uv}\cup E_u^v) & g(u,\Delta'''_i,p) & g(u,\Delta'''_i,p) & \ast \\
 & \vdots & \vdots & \vdots \\\cline{2-4}
 & \vdots & \vdots & \vdots \\
\Delta''''_i\in E_v\setminus (E_{uv}\cup E_v^u) & 0 & g(v,\Delta''''_i,p) & \ast \\
 & \vdots & \vdots & \vdots \\\cline{2-4} 
E(H-u-v) & 0 & 0 & R(H-u-v,p|_{V-u-v})\\
},
\]
where we notice that the entries of $V(H)\setminus \{u\}$ in the row block of $E_u^v$ become zero by 
(\ref{eq:vsplitting1}).

Since $p(u)=p(v)$, we have $g(u,\Delta'''_i,p)=g(v,\Delta'''_i-u+v,p)$ for 
$\Delta'''_i\in E_u\setminus (E_{uv}\cup E_u^v)$.
Hence, the bottom-right block
\[
\kbordermatrix{
 & v &  \\
 & \vdots & \vdots  \\
 \Delta''_i\in E_v^u &   g(v,\Delta''_i,p) & \ast \\
  & \vdots  & \vdots \\\cline{2-3}
  & \vdots  & \vdots \\
 \Delta'''_i\in E_u\setminus (E_{uv}\cup E_u^v)  & g(u,\Delta'''_i,p) & \ast \\
 &  \vdots & \vdots \\\cline{2-3}
  &  \vdots & \vdots \\
 \Delta''''_i\in E_v\setminus (E_{uv}\cup E_v^u)  & g(v,\Delta''''_i,p) & \ast \\
  &  \vdots & \vdots \\\cline{2-3} 
 E(H-u-v)  & 0 & R(H-u-v,p|_{V-u-v}) 
},
\]
is exactly equal to $R(H/uv,q)$ by identifying $\Delta'''_i$ with $\Delta'''_i-u+v$
for $\Delta'''_i\in E_u\setminus (E_{uv}\cup E_u^v)$.
Thus, the rank of $R^*$ is equal to the rank of the matrix
\[
\kbordermatrix{
 & u  &  \\
 & \vdots & \vdots \\
 \Delta_i\in E_{uv} & \bd_{\Delta_i} & {\bf 0} \\
 & \vdots &\vdots\\\cline{2-3}
 & \vdots & \vdots \\
\Delta'_i\in E_{v}^u & g(u,\Delta'_i,p) & {\bf 0}  \\
 & \vdots & \vdots \\\cline{2-3}
& *  & R(H/uv,q)
}
 =\begin{bmatrix}
 A_{uv}(H,q,\bd)&{\bf 0}\\
 *&R(H/uv,q)
 \end{bmatrix} \,.
\renewcommand{\arraystretch}{1}
\]
Hence $\rank R^*\geq \rank R(H/uv,q)+\rank A_{uv}(H,q,\bd)=d(|V|-1)-{{d+1}\choose2}+d=d|V|-{{d+1}\choose2}$ since $(H/uv,q)$ is infinitesimally volume rigid and $\rank A_{uv}(H,q,\bd)=d$.

Define $p_t:V\to \R^{d}$ for $t\in \R$ 
by putting  $p_t(u)=p(v)+t\bd$ and $p_t(w)=p(w)$ for all $w\neq u$. 
Then $p_t(u)-
p_t(u)^{\Delta_i}=t\bd_{\Delta_i}$ for each 
 $\Delta_i\in E_{uv}$. 
We will complete the proof by showing that, after a suitable scaling, 
the rows of 
$R(H,p_t)$ become arbitrarily close to  the rows of $R^*$ when $t$ is sufficiently close to zero.

Consider the matrix $R_t$ obtained from $R(H,p_t)$ by dividing the row
labelled by each $\Delta_{i}\in E_{uv}$ by $t\Vol (p(\Delta_{i}-u))^2$ when 
$|\Delta_i|\geq 3$ and by $t$  when $|\Delta_i|=2$. We will show that $\lim_{t\to 
0}R_t=R^*$. It is easy to see that this holds for all rows of  $R_t$ and $R^*$ other 
than the rows labelled by the hyperedges $\Delta_{i}\in E_{uv}$ with $|\Delta_i|\geq 3$.

The vector in the row of $R_t$ labelled by any such $\Delta_{i}$ and the columns labelled by $u$ is 
$$\frac{g(u,\Delta_{i},p_t)}{t\Vol (p(\Delta_{i}-u))^2}=\frac{(p_t(u)-p_t(u)^{\Delta_{i}})\Vol (p_t(\Delta_{i}-u))^2}{t\Vol(p(\Delta_{i}-u))^2}=\frac{p_t(u)-p_t(u)^{\Delta_{i}}}{t}=\bd_{\Delta_i},
$$
since $p_t(\Delta_{i}-u)=p(\Delta_{i}-u)$. 
Thus the  entries in the rows of $\lim_{t\to 0}R_t$ and $R^*$ labelled by $\Delta_{i}$ and the columns labelled by $u$ are the same.

For each $w\in \Delta_i-u-v$, the vector in the row of $R_t$ labelled by $\Delta_{i}$ and the columns labelled by $w$ is 
\begin{align*}
    \frac{g(w,\Delta_{i},p_t)}{t\Vol (p(\Delta_{i}-u))^2}&=\frac{(p_t(w)-p_t(w)^{\Delta_{i}})\Vol (p_t(\Delta_{i}-w))^2}{t\Vol (p(\Delta_{i}-u))^2}\\
    &=\frac{(p(w)-p_1(w)^{\Delta_{i}})}{\Vol (p(\Delta_{i}-u))^2}\,\frac{\Vol (p_t(\Delta_{i}-w))^2}{t},
\end{align*}
since $p_t(w)=p(w)$ and $p_t(w)^{\Delta_{i}}=p_1(w)^{\Delta_{i}}$ (because the affine span of $p_t(\Delta_{i}-w)$ is the same for all $t\neq 0$). We can now use the fact that  
\begin{align*}
(k-2)!\Vol (p_t(\Delta_{i}-w))^2&=\|p_t(u)-p_t(u)^{\Delta_{i}-w}\|^2\Vol (p_t(\Delta_{i}-w-u))^2\\
&=t^2\|p_1(u)-p_1(u)^{\Delta_{i}-w}\|^2\Vol (p(\Delta_{i}-w-u))^2
    \end{align*}
to deduce that $\lim_{t\to 0} \frac{g(w,\Delta_{i},p_t)}{t\Vol p(\Delta_{i}-u)^2}=\bf 0$. Thus the  entries in the rows of $\lim_{t\to 0}R_t$ and $R^*$ labelled by $\Delta_{i}$ and the columns labelled by $w$ are all zero.

It remains to show that entries in the row of $R_t$ labelled by $\Delta_{i}$ and the columns labelled by $v$ converge to $-\bd_{\Delta_i}$. This follows from the  previous two paragraphs and the observation that the sum of the vectors labelled by the vertices of $V$ in each row of $R(H,p_t)$  is equal to the zero vector (since any translation of $\R^{d}$ preserves the volume of every hyperedge in $H$).

Hence $\lim_{t\to 0}R_t=R^*$.
We can now choose a $t>0$ such that $\rank R_t=\rank R^*$. This gives
$
    \rank R(H,p_t)=\rank R_t=\rank R^*\geq d|V|-{{d+1}\choose2}
$
so $(H,p_t)$ is infinitesimally volume rigid. Since the rank of the volume rigidity matrix will be  maximised at any generic realisation, this implies that $H$ is volume rigid in $\R^d$.
\end{proof}

In our particular applications in Section \ref{sec:jface rigid}, we verify the full rank condition on $A_{uv}(H,q,\bd)$ given in Lemma~\ref{lem:vsplit_kface} by exhibiting special realisations which have full rank. It would be useful to have a simple combinatorial  condition which guaranteed that $A_{uv}(H,q,\bd)$ had full rank.

\subsection{A coning lemma}
We will obtain a sufficient condition for the addition of a new vertex and a set of new hyperedges to a hypergraph which is volume rigid in $\R^d$, to result in a hypergraph which is volume rigid in $\R^{d+1}$. Our result extends Whiteley's coning lemma for bar-joint rigidity of graphs to volume rigidity of hypergraphs.

Given a graph $G=(V,E)$ and a vertex $v\not\in V$,
the {\em cone} of $G$ is the graph
$G*v$ obtained by adding the new vertex $v$ and all edges from $v$ to $V$.
A fundamental result in bar-joint rigidity due to Whiteley \cite{Wcone} tells us that 
$G*v$ is rigid in $\mathbb{R}^{d+1}$
if and only if $G$ is  rigid in $\mathbb{R}^d$.
Whiteley's proof technique for sufficiency is to start with an infinitesimally  rigid realization $p$  of $G$ in $\mathbb{R}^{d}$ and extend this to a realisation $p^*$ of $G*v$ in $\R^{d+1}$ by putting the {\em cone vertex} $v$ on the new coordinate axis.
He then sends the cone vertex to infinity along this axis and shows that the rigidity matrix of $(G*v,p^*)$ converges
(after an appropriate scaling and reordering of the columns so that the first $n$ columns correspond to the $(d+1)$-coordinates of the entries in the standard rigidity matrix of $(G*v,p^*)$) to the block-diagonal form
\[
\kbordermatrix{
 & v & V \\
 & I_n & 0 \\
 & 0 & R(G,p)
},
\]
where $I_n$ denotes the identity matrix of size $|V|=n$ and $R(G,p)$ denotes the  rigidity matrix of the original framework. 
This block-diagonal form enables us to confirm the existence of an infinitesimally rigid realization of $G*v$ in $\R^{d+1}$.

We shall extend Whiteley's proof technique to volume rigidity of hypergraphs.
In this case, the top-left block of the above block-diagonal form becomes more complicated so we will begin by defining it separately.

Let $H=(V,E)$ be a hypergraph with hyperedges of size at most $d+1$ and $p:V\rightarrow \mathbb{R}^d$ be a point configuration such that (the multiset) $p(\Delta)$ is affinely independent in $\mathbb{R}^d$ for all $\Delta\in E$.
For each $\Delta\in E$, let $p_{\Delta}$ be the projection of the origin to the affine span of $p(\Delta)$.
Since $p(\Delta)$ is affinely independent,
$p_{\Delta}=\sum_{x\in \Delta}\alpha_x p(x)$ for  unique scalars $\alpha_x\in \R$ with
$\sum_{x\in \Delta}\alpha_x=1$.  
Let $A(H,p)$ be the $|E|\times |V|$ matrix in which the entry in row $\Delta\in E$ and column $x\in V$ is $\alpha_x$ if $x\in \Delta$ and is zero otherwise. Let ${\cal A}_d(H)$ be the row matroid of $ A(H,p)$ for any generic $p$. This matroid is uniquely defined since, by Lemma \ref{lem:rational},  the entries of $ A(H,p)$ are rational functions of the coordinates of $p$. 

Recall that, for $w\in V$, $E_w$ denotes the set of hyperedges containing $w$.
We denote  the hypergraph $(V-w,\{\Delta-w: \Delta\in E_w\})$ by $H_w$. Note that $H_w$ has $|V|-1$ vertices so $\rank\cA_d(H_w)\leq |V|-1$.

\begin{lemma}\label{lem:v_add} Let $d\geq 3$ be an integer, $H=(V,E)$ be a hypergraph with hyperedges of size at most $d+1$
and $w\in V$. 
Suppose  that $H-w$ is  volume rigid in $\R^{d}$ and $\rank \cA_d(H_w)=|V|-1$. Then $H$ is  volume rigid in $\R^{d+1}$. 
\end{lemma}
\begin{proof}
 Let $q$ be a generic realisation of $H$ in $\R^d$.
 For $t \in \mathbb R$, let $p_t: V \to \mathbb R^{d+1}$ be defined by $p_t(v) = (q(v),0) $ for $v\in  V- w$ and $p_t(w) = (q(w),t)$. 
 We will determine the behaviour of the entries of $R(H,p_t)$  as $t \to \infty$. {We will assume throughout this proof that $|\Delta|\geq 3$ for all $\Delta\in E_w$. Hyperedges of size two can be dealt with in the same way as in the coning lemma for graphs.  The argument is more complicated for larger hyperedges and we believe that mixing the two arguments will distract from the main ideas in our proof.} 

Choose a hyperedge $\Delta\in E_w$. Then $q(\Delta)$ is affinely independent and hence
 $q(w)^{\Delta}=\sum_{x\in \Delta-w}\alpha_x q(x)$ for unique scalars $\alpha_x$
 with $\sum_{x\in \Delta-w}\alpha_x=1$.
\begin{claim}\label{clm:alpha}
For each $u\in \Delta-w$, 
$q(w)^{\Delta}-q(w)^{\Delta-u}=\alpha_u\left( q(u)-q(u)^{\Delta-w}\right)$.
\end{claim}
 \begin{proof}
 Let $\pi$ be the orthogonal projection of $\mathbb{R}^d$ onto the affine span of $q(\Delta-u-w)$.
 Then $q(w)^{\Delta}-\pi(q(w)^{\Delta})=\alpha_u\left( q(u)-\pi(q(u))\right)$.
 The claim follows by observing that $\pi(q(w)^{\Delta})=q(w)^{\Delta-u}$ and $\pi(q(u))=q(u)^{\Delta-w}$.
 \end{proof}

Let $I_d$ denote the $d \times d$ identity matrix.
 \begin{claim} \label{clm:limit}
        For each $u \in \Delta- w$,
    $$\lim_{t \to \infty} \left(R(H,p_t)_{\Delta,u}\begin{pmatrix}
        \frac{1}{t^2}I_{d} & 0 \\ 0 & -\frac1t 
    \end{pmatrix}\right) = \left( \frac1{(|\Delta|-2)^2} R(H_w,q|_{V-w})_{\Delta-w,u}, \Vol(q(\Delta-w))^2 \alpha_u \right).$$
\end{claim}   

\begin{proof}
    We have $q(\triangle) \subset \R^d \subset \R^{d+1}$ where the last inclusion is induced by $x \in \R^d \mapsto (x,0) \in \R^{d+1}$. Let 
    $\pi:\mathbb R^{d+1} \to \aff (q(\Delta-w))$ 
    be the orthogonal projection of $\mathbb R^{d+1}$ onto the affine subspace spanned by $q(\Delta-w)$. Note that for all $v \in \Delta-w$ and all $t \in \R$, $p_t(v) = (q(v),0)$.
    An elementary argument using similarity of triangles (see Figure \ref{fig:sim1}) and the fact that $p_t(u)_{d+1}=0$ gives
    \begin{equation}
        \label{eqn:sim}
        {\rm sign}(\alpha_u)\frac{\| p_t(u) - \pi(p_t(u)^\Delta)\|}{ (p_t(u)^\Delta)_{d+1}}
        =\frac{t}{\| p_t(w)^\Delta  - p_t(w)^{\Delta-u}\|},
    \end{equation}
    where ${\rm sign}(\alpha_u)$ denotes the sign of $\alpha_u$,
    which, for large $t$, is determined by which side of 
    {${\rm aff}(p_t(\Delta-u))$ the point  $p_t(u)$ belongs to in ${\rm aff}(p_t(\Delta))$. Both sides of \eqref{eqn:sim} are equal to the tangent of the angle between the hyperplanes ${\rm aff}(p_t(\Delta-u))$ and ${\rm aff}(p_t(\Delta-w))$ in ${\rm aff}(p_t(\Delta))$.} 
\begin{figure}
    \centering
\definecolor{wewdxt}{rgb}{0.43137254901960786,0.42745098039215684,0.45098039215686275}
\definecolor{ududff}{rgb}{0.30196078431372547,0.30196078431372547,1}
\begin{tikzpicture}[line cap=round,line join=round,>=triangle 45,x=0.9cm,y=0.9cm]
\clip(-11.1,-2.446942148760324) rectangle (1.412142994426292,6.071220449740549);
\draw [line width=1.5pt,domain=-10.211920046127254:1.412142994426292] plot(\x,{(-7-0*\x)/7});
\draw [line width=1.5pt,domain=-10.211920046127254:1.412142994426292] plot(\x,{(--51--6*\x)/3});
\draw [line width=1pt,dash pattern=on 1pt off 4pt] (-2,-1)-- (-7.6,1.8);
\draw [line width=1pt,dash pattern=on 1pt off 4pt] (-7.6,1.8)-- (-7.6,-1);
\draw (-7.3,0.4) node {$h$};
\draw [line width=1pt,dash pattern=on 1pt off 4pt] (-6,5)-- (-6,-1);
\draw (-5.7,2) node {$t$};
\begin{scriptsize}
\draw [fill=black] (-9,-1) circle (2.5pt);
\draw (-10,-0.7) node {$p_t( \Delta - u - w)$};
\draw [fill=black] (-2,-1) circle (2.5pt);
\draw (-1.9858965981164773,-1.3) node {$p_t(u)$};
\draw [fill=black] (-6,5) circle (2.5pt);
\draw (-6.4,5.140987891601014) node {$p_t(w)$};
\draw [fill=black] (-7.6,1.8) circle (2pt);
\draw (-8,2.0658389390736227) node {$p_t(u)^\Delta$};
\draw [fill=black] (-7.6,-1) circle (2pt);
\draw[color=black] (-7.628794926004244,-1.3) node {$\pi\left( p_t(u)^\Delta \right)$};
\draw [fill=black] (-6,-1) circle (2pt);
\draw[color=black] (-5.845208533538356,-1.3) node {$p_t(w)^\Delta$};
\end{scriptsize}
\end{tikzpicture}
    \caption{An orthogonal projection of $\aff(p_t(\Delta))$ to a 2-dimensional affine subspace orthogonal to  $\aff(p_t(\Delta-u-w))$. Points are labeled by their pre-images in $\aff(p_t(\Delta))$. The affine subspace $\aff(p_t(\Delta-w))$ maps to the horizontal line and the affine subspace $\aff(p_t(\Delta-u))$ maps to the oblique line. The lengths of the vertical dotted segments are $t$ and $h = (p_t(u)^\Delta)_{d+1} = \|p_t(u)^\Delta - \pi(p_y(u)^\Delta\|$ (assuming here that $t >0$).}
    \label{fig:sim1}
\end{figure}
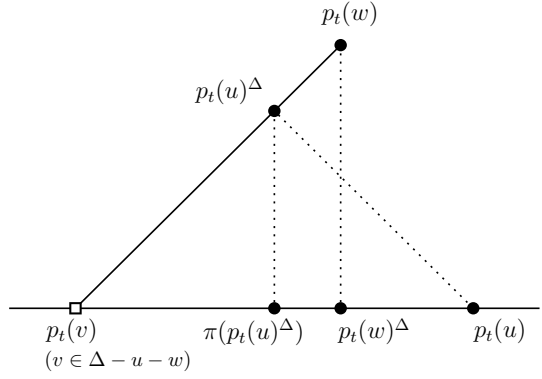

We also have
\begin{equation}
    \label{eqn:v}\Vol(p_t(\Delta-u))^2 = \frac{1}{(|\Delta|-2)^2}\|p_t(w) - p_t(w)^{\Delta-u}\|^2\Vol(p_t(\Delta-w-u))^2.
\end{equation} 
Hence,
\begin{align}
    &\frac{(R(H,p_t)_{\Delta,u})_{d+1}}{t}
    =\frac{\Vol(p_t(\Delta-u))^2}{t} (p_t(u)-p_t(u)^{\Delta})_{d+1} \nonumber\\
    &=-{\rm sign}(\alpha_u) \frac{\Vol(p_t(\Delta-u))^2}{t^2} \|p_t(u) -\pi(p_t(u)^\Delta)\| \| p_t(w)^\Delta - p_t(w)^{\Delta-u}\|   \nonumber\\
    &= -{\rm sign}(\alpha_u) \frac{\Vol(q(\Delta-u-w))^2}{ t^2 (|\Delta|-2)^2} \|p_t(w) - p_t(w)^{\Delta-u}\|^2 \|p_t(u) -\pi(p_t(u)^\Delta)\| \| p_t(w)^\Delta - p_t(w)^{\Delta-u}\| \label{eqn:vol},
\end{align}
where the first equation is by definition, 
the second equation follows from $(p_t(u))_{d+1}=0$ and (\ref{eqn:sim}),
and the third equation follows from (\ref{eqn:v}).

We now send $t$ to $\infty$. 
The fact that $p_t(w) - (q(w),0)$  is orthogonal to $\aff(p_t(\Delta-w))$ gives
\begin{equation}
    \label{eqn:lim}
    \lim_{t \to \infty}
    \begin{pmatrix}
p_t(w)^{\Delta} \\ 
p_t(w)^{\Delta-u} \\
\pi(p_t(u)^{\Delta})       
    \end{pmatrix}
    \rightarrow
    \begin{pmatrix}
        (q(w)^{\Delta},0)\\
 (q(w)^{\Delta-u},0)\\
 (q(u)^{\Delta-u},0)
    \end{pmatrix}
\end{equation}
and hence
\begin{equation}
    \label{eqn:r}\lim_{t \to \infty} 
    \frac{\|p_t(w) - p_t(w)^{\Delta-u}\|}{t} = 1. 
\end{equation}
We can combine \eqref{eqn:vol}, \eqref{eqn:lim} and \eqref{eqn:r} to obtain 
\begin{align*}
    &\lim_{t \to \infty} \frac{(R(H,p_t)_{\Delta,u})_{d+1}}{t}  \\
    &=-{\rm sign}(\alpha_u)  \frac{\Vol(q(\Delta-u-w))^2}{ (|\Delta|-2)^2} \|q(u) -q(u)^{\Delta-w}\| \| q(w)^\Delta - q(w)^{\Delta-u}\| \notag\\
    &=  -\frac{\Vol(q(\Delta-u-w))^2}{ (|\Delta|-2)^2} \|q(u) -q(u)^{\Delta-w}\|^2  \alpha_u \qquad (\text{by Claim~\ref{clm:alpha}})  \\
    &= - \Vol(q(\Delta-w))^2 \alpha_u 
\end{align*}
which verifies that equality holds for  the $(d+1)$-th component in the statement of the claim. 

In addition we 
deduce that
\begin{align*}
    \label{eqn:o}
    \lim_{t \to \infty} \frac{R(H,p_t)_{\Delta,u}}{t^2} &= \lim_{t \to \infty } \frac{\Vol(p_t(\Delta-u))^2}{t^2} (p_t(u)-p_t(u)^\Delta) & \\
    &= \lim_{t \to \infty } \frac{\Vol(p_t(\Delta-w-u))^2\|p_t(w) - p_t(w)^{\Delta-u}\|^2}{(|\Delta |-2)^2t^2} (p_t(u)-p_t(u)^\Delta) & \text{using (\ref{eqn:v})}\\
    &= \frac{\Vol(q(\Delta-w-u))^2}{(|\Delta|-2)^2} (q(u)-q(u)^{\Delta-w},0) & \text{using (\ref{eqn:lim}) and (\ref{eqn:r})} \\
    &= \left(\frac{1}{(|\Delta|-2)^2} R(H_w,q|_{V-w})_{\Delta-w,u},0\right) &  
\end{align*}
This completes the proof of the claim.
\end{proof}

Let $R_t$ be the matrix obtained from $R(H,p_t)$ by deleting the $(d+1)$ columns associated with $w$ and by permuting columns so that the last $|V|-1$ columns consist of the columns indexed by the $(d+1)$-th coordinates of the vertices in $V-w$.
Since $p_t(v)=(q(v),0)$ for $v\in V-w$, 
$R_t$ can be written as 
\[
R_t=\kbordermatrix{
 &\text{first $d$ coordinates of $V-w$} & \text{$(d+1)$-th coordinates of $V-w$}\\
E_w & B_t & A_t \\
E\setminus E_w & R(H-w,q|_{V-w}) & 0 
}
\]
for some $A_t$ and $B_t$. 
Since $\rank R(H-w, q|_{V-w})=d(|V|-1)-{d+1\choose 2}$ by the hypothesis that  $H-w$ is volume rigid in $\R^d$, it suffices to show that $\rank A_t=|V|-1$ for some $t$. 

 By Claim~\ref{clm:limit},
 $\lim_{t\rightarrow \infty} -\frac1t A_t 
 =D A(H_w,q|_{V-w})$,
 where  $D$ denotes the $|E_w| \times |E_w|$ diagonal matrix 
 with diagonal entries $\Vol(q(\Delta-w))^2$ for $\Delta \in E_w$. 
 Since $q$ is in general position, $D$ is non-singular.
 The hypothesis that $\rank \cA_d(H_w)=|V|-1$ now gives $\rank D A(H_w,q|_{V-w}) =|V|-1$.
 Thus, for all but finitely many values of $t$, $\rank A_t=|V|-1$.
 This completes the proof.
\end{proof}

\medskip
\noindent
{\bf Remarks.}
A more careful analysis  in the proof of Lemma~\ref{lem:v_add} tells us that:\\
(a)
when $|\Delta|\geq 3$ for each $\Delta\in E_w$, the matrix $R_t$ satisfies
{
\begin{equation}
    \label{eqn:limit}
    \lim_{t\to \infty} \begin{pmatrix}
        I_{|E_w|} & 0 \\ 0 & t^2 I_{|E\sm E_w|} 
    \end{pmatrix}R_t \begin{pmatrix} \frac1{t^2}I_{d|V|} & 0 \\ 0 & \frac1t I_{|V|-1}  \end{pmatrix}  = \begin{pmatrix}
        D'R(H_w,q|_{V-w}) & DA(H_w,q|_{V-w})  \\
        R(H-w,q|_{V-w}) & 0   
    \end{pmatrix}
\end{equation}
}

\noindent 
where 
$D,D'$ denote the $|E_w| \times |E_w|$ diagonal matrices 
with diagonal entries $\Vol(q(\Delta-w))^2$ and $(|\Delta|-2)^2$ for $\Delta \in E_w$, respectively.
Hence, even if  $H-w$ is not volume rigid in $\mathbb{R}^d$, we may still be able to show that $H$ is volume rigid in $\mathbb{R}^{d+1}$  by analysing the block matrix on the right hand side of  (\ref{eqn:limit}).\\[1mm]
(b) any infinitesimally volume rigid realisation of $H-w$ in $\R^d$ can be extended to an infinitesimally rigid realisation of $H$ in $\R^{d+1}$ by placing $w$ at all but finally many points on the $(d+1)$'th coordinate axis.

\medskip

We next obtain a sufficient condition for the rank of the matroid $\cA_d(H)$ of a $k$-uniform hypergraph $H=(V,E)$ to be equal to $| V|$. 
A $k$-uniform hypergraph $H$ is {\em strongly connected} if, 
for any two distinct $\Delta,\Delta'\in E$, there exists a sequence of hyperedges $\Delta_1,\Delta_2,\ldots,\Delta_m$ such that $\Delta_1=\Delta$, $\Delta_m=\Delta'$ and
$|\Delta_i\cap \Delta_{i+1}|=k-1$ for all $1\leq i\leq m-1$. We will refer to the maximal strongly connected sub-hypergraphs of $H$ as the {\em strongly connected components} of $H$. It is easy to see that the property of belonging to the same strongly connected component is an equivalence relation on $E$ and hence the edge sets of the  strongly connected components of $H$ partition $E$.

\begin{lemma}
    \label{lem:alt1}
    Let $H=(V,E)$ be a $k$-uniform hypergraph. 
    Suppose that  
    every strongly connected component of $H$ contains a copy of $K_{k+1}^{k}$. Then $\rank \cA_d(H) = |V|$ for all $d\geq k$. 
\end{lemma}

\begin{proof}
    Let $p$ be a generic realisation of $H$ in $\R^d$. It will suffice to show that $\rank A(H,p)=|V|$.
    
    We first consider the case when $H$ is strongly connected.
    Suppose that $V = \{v_1,\cdots,v_n\}$ and that $Z = \{v_1,\ldots,v_{k+1}\}$ induces a copy of $K_{k+1}^{k}$ in  $H$.  Let $Y = \{\Delta_1,\ldots,\Delta_{k+1}\}$ be the set of hyperedges of  $H[Z]$, ordered so that $\Delta_i=V-v_i$ for $1\leq i\leq k+1$. Let $B(p)$ be the $(k+1) \times (k+1)$ submatrix of $A(H,p)$ corresponding to the rows indexed by $Y$ and the columns  indexed by $Z$. 
    \begin{claim}
        \label{clm:B}
        $B(p)$ is a non-singular matrix.
    \end{claim}
    \begin{proof}
         Observe that the entries of $B(p)$ are rational functions of $p$ and since $p$ is generic it suffices to find a special position $q$ such that $B(q)$ is non-singular. Choose any $q:V(\cS) \to \mathbb R^k$ such that $q(v_1),\ldots,q(v_{k+1})$ are the vertices of a regular $k$-dimensional simplex in $\mathbb R^d$, centered on the origin. Then
         $B(q) = \frac1{k+1} (J - I)$ where $J$ is the $(k+1)\times (k+1)$ matrix of all 1's and $I$ is the identity matrix, which is a non-singular matrix. 
\end{proof}

    Now, since $H$ is strongly connected, after a suitable relabelling of the vertices $v_{k+2},\ldots,v_{n}$, we can choose 
    $\Delta_{k+2},\ldots, \Delta_{n}\in E(H_w)$ such that, for
    $k+2 \leq m \leq n$, $\bigcup_{i =1}^m \Delta_i = \{v_{1},\ldots,v_m\}$. For $k+1 \leq m \leq n$, let $C_m$ be the $i \times i$ submatrix of $A(H,p)$ induced by the rows indexed by $\Delta_1,\ldots,\Delta_m$ and columns indexed by $v_1,\ldots,v_m$. Observe that for $k+1\leq m \leq n-1$, we have a block decomposition 
    $$C_{m+1} = 
    \kbordermatrix{
    &v_1 & \cdots & v_m & v_{m+1} \\
    \Delta_1 & & & & 0 \\
    \vdots & & C_m & & \vdots \\
    \Delta_m & & & & 0 \\
    \Delta_{m+1} & \star & \cdots & \star & \alpha_{v_{m+1}} \\
    }
    $$
    Also $\alpha_{v_{m+1}} \neq 0$ for each $m$ since $p$ is generic. It follows that, if $C_m$ is non-singular, then $C_{m+1}$ is also non-singular. Since $C_{k+1}=B(p)$, $C_{k}$ is non-singular by Claim \ref{clm:B}. Therefore $C_n$ is non-singular and so $\rank A(H,p) = n$. This completes the proof in the case that $H$ is strongly connected.
    
    It remains to consider the case when $H$ is not strongly connected.
    Let $\lambda$ be an element of the right kernel of $A(H,p)$ and  $v_i \in V$. Since every strongly connected component of $H$ contains a copy of $K_{k+1}^k$, every vertex in $V$ belongs to a hyperedge in $E$ and hence we can choose a strong component $H_i$ of $H$ such that $v_i \in V(H_i)$. Suppose that $|V(H_i)| = m$. Let $A_i$ be the submatrix of $A(H,p)$ consisting of the rows labelled by $E(H_i)$. Then the entries in every  column of $A_i$ labelled by a vertex in $V\sm V(H_i)$ are zero. In addition, since $H_i$ is strongly connected and contains a copy of $K_{k+1}^{k}$, $\rank(A_i) = m$ by the preceding paragraph. Since $A(H,p)\lambda = 0$ it follows that $\lambda_{v_i} = 0$ and, since $v_i\in V$ was arbitrary, $\lambda = 0 $. Hence $\rank A(H,p)= n$.
\end{proof}

\section{Volume  rigidity of simplicial complexes}\label{sec:jface rigid}

We will prove Theorem \ref{thm:d-3} and  show that Conjecture \ref{con:d-2} holds when $d=4,5,6$.

 We first need to introduce some terminology.  
 Recall that an  {\em (abstract) simplicial complex} $\cS$ is a family of sets which is closed under 
 inclusion. We refer to the elements of $\cS$ of cardinality $k+1$ as {\em $k$-faces} and the maximal elements of $\cS$ as {\em facets}. We use $V(\cS)$ and $E(\cS)$ to denote the set of {\em vertices}, i.e. $0$-faces, and {\em edges}, i.e. 1-faces, of $\cS$, respectively. Given $X\in \cS$, the {\em link}  of $X$ in $\cS$ is the complex 
 $$\lk(X)= \lk_\cS(X) = \{Y:X\cap Y=\emptyset \mbox{ and } X\cup Y\in \cS\}$$ and the {\em star}  of $X$ is the complex 
 $$\st(X)= \st_\cS(X) = \{X'\cup Y:X'\subseteq X \mbox{ and } Y\in \lk(X)\}.$$  We say $\cS$ is a {simplicial $d$-complex} if each of its facets has cardinality $d+1$.
 
Recall that the {\em topological realisation} of a simplicial complex $\cS$ is the subset of $\R^{V(\cS)}$ defined by 
$$|\cS| = \bigcup_{X \in \cS} \conv\{\be_v: v \in X\}$$
where, for a vertex $v \in V(\cS)$, $\be_v$ is the corresponding standard basis vector of $\R^{V(\cS)}$.
We say that $\cS$ is a {\em  (connected) simplicial $d$-manifold} if $|\cS|$ is a connected $d$-manifold.  
 
We will refer to the  $(k+1)$-uniform hypergraph  whose hyperedges are the $k$-faces of a simplicial complex $\cS$ as the  {\em $k$-skeleton hypergraph} of $\cS$ and denote it by $H_k(\cS)$. A simplicial $d$-complex is said to be {\em strongly connected} if its $d$-skeleton hypergraph is strongly connected.

Theorem~\ref{thm:d-3} follows immediately from our next result and the above mentioned result of Kalai~\cite[Theorem 1.2]{K} (which tells us that the 1-skeleton of a simplicial $(d-1)$-manifold is rigid in $\R^d$).
\begin{theorem}
\label{thm:new}
    Let $\cS$ be a simplicial $t$-complex such that $H_{l}(\cS)$  is volume rigid in $\R^d$, for some $1 \leq l \leq t \leq d-1$. Then 
    $H_k(\cS)$ is volume rigid in $\R^d$ for all $1\leq k\leq t-2$.
\end{theorem}
\begin{proof}
Let $p:V(\cS) \to \R^d$ be generic and suppose that that $\dot{p}$ is an infinitesimal motion of $(H_{k}(\cS),p)$.

Let $\Delta$ be a $t$-face of $\cS$ and let
$K$ be the copy of $K_{t+1}^{k+1}$ consisting of the $k$-faces (i.e. $(k+1)$-subsets) of $\Delta$. Then, since $k \leq t-2$,
Lemma \ref{lem:simplex} implies that $(K,p|_\Delta)$ is volume rigid in $\R^d$ and so $\dot{p}$ must infinitesimally preserve the volume of  every $l$-face of $(\cS,p)$ which is contained in $\Delta$. Since every $l$-face of $\cS$ is contained in some $t$-face of $\cS$ this implies that  $\dot{p}$
must also be an infinitesimal motion of $(H_l(\cS),p)$ and so $\dot{p}$ is a trivial infinitesimal motion as required.
\end{proof}

\begin{proof}[Proof of Theorem \ref{thm:d-3}]
    Recall that $H = H_k(\cS)$ where $\cS$ is a connected simplicial $(d-1)$-manifold for some $d\geq 4$ and $1 \leq k \leq d-3$. Now the theorem follows from Kalai's Theorem \cite[Theorem 1.2]{K} and Theorem \ref{thm:new} with $l=1$ and $t =d-1$. 
\end{proof}

Combined with Theorem \ref{thm:d-3}, Conjecture \ref{con:d-2} would imply that the $k$-skeleton hypergraph of a connected simplicial $(d-1)$-manifold is volume rigid in $\R^d$ for all $1\leq k\leq d-2$. We  first point out that the same conclusion holds for strongly connected simplicial $d$-complexes.

\begin{theorem}\label{thm:weak}
Let $d\geq 3$ be an integer and $\cS$ be a strongly connected simplicial $d$-complex. Then the $k$-skeleton hypergraph of $\cS$ is volume rigid in $\R^{d}$ for all $1\leq k\leq d-2$. 
\end{theorem}
\begin{proof}
Let $H=H_{k}(\cS)$.
Suppose, for a contradiction, that 
$H$ is not volume rigid in $\R^{d}$. Let $\cT$ be a maximal simplicial $d$-complex contained in $\cS$ such that  $H_{k}(\cT)$ is volume rigid in $\R^{d}$. Since $\cS$ is 
strongly connected, we can find two $d$-simplices $X_1$ and $X_2$ with $X_1\in \cT$, $X_2\in \cS\sm\cT$ and $|X_1\cap X_2|\geq d$. Then $H[X_2]$ is volume rigid by Theorem \ref{thm:complete} and we can now deduce that $H_k(\cT)\cup H[X_2]$
is  volume rigid in $\R^{d}$ by Lemma \ref{lem:glue_k}. This contradicts the maximality of $\cT$ and completes the proof of the theorem.
\end{proof}

We next verify Conjecture \ref{con:d-2} when $d=4,5,6$. Our inductive proof 
naturally leads us to consider a larger family of simplicial complexes than simplicial manifolds. Following Kalai \cite{K}, we define a {\em simplicial homology $k$-manifold} to be a simplicial $k$-complex $\cS$ with the property that, for all $0\leq j\leq k-1$ and all $j$-faces $X$ of $\cS$, the simplicial homology (with coefficients in $\mathbb Z$) of $\lk_\cS(X)$  is the same as that of a simplicial $(k-j-1)$-sphere. We say that $\cS$ is a {\em simplicial homology $k$-sphere} if, in addition, $\cS$  has the same
simplicial homology as a simplicial $k$-sphere. We will  use the fact that every simplicial homology $2$-sphere is a simplicial 2-sphere.
\begin{lemma}
    Let $\cS$ be a connected homology $k$-manifold for some $k \geq 1$. Then $H_k(\cS)$ is strongly connected. 
\end{lemma}
\begin{proof}
    For $k = 1$, $H_1(\cS)$ is a connected graph each of whose vertices has degree 2. Hence, $H_1(\cS)$ is a cycle graph and the lemma is true in this case. So we can assume from now on that $k\geq 2$.
    Suppose that $H_k(\cS)$ is not strongly connected. Let $K$ be a maximal strongly connected set of $k$-faces of $\cS$ and let $L$ be the complement of $K$ in the set of $k$-faces of $\cS$. Let $\cS[K] = \{ F \in \cS : F \text{ is contained in some element of $K$}\} $ and define $ \cS[L]$ analogously. By our assumption $L \neq \emptyset$ and so $\cS[K]$ and $\cS[L]$ are both non-empty proper subcomplexes of $\cS$ such that $\cS[K] \cup \cS[L] = \cS$. Since $\cS$ is connected, there is a non-empty face in $\cS[L] \cap \cS[K]$. Let $F$ be a maximal face in $\cS[L] \cap \cS[K]$. Then $F$ cannot be a $(k-1)$-face since, by hypothesis, every $(k-1)$-face of $\cS$ belongs to exactly two $k$-faces, and if one of these lies in $K$ and the other in $L$ this would contradict the maximal choice of $K$. Therefore the dimension of $F$ is at most $k-2$. 
    
    Now $\lk_\cS(F) = \lk_{\cS[K]}(F) \cup \lk_{\cS[L]}(F)$ and, by the maximal choice of $F$,  $\lk_{\cS[K]}(F) \cap \lk_{\cS[L]}(F) = \{\emptyset\}$. It follows that $\lk_\cS(F)$ is not connected. But since $\cS$ is a homology $k$-manifold, the only faces that have a disconnected link are $(k-1)$-faces, contradicting our earlier deduction.
\end{proof}

We also need the following construction. Given 
a simplicial complex  $\cS$ and a finite set  of vertices $Z$ disjoint from $V(\cS)$, the {\em simplicial cone of $\cS$ with $Z$} is the complex 
$$\cS*Z=\{X\cup Z':X\in \cS \mbox{ and } Z'\subseteq Z\}.$$

\begin{lemma} \label{lem:cone}
Let  $\cS$ be a simplicial 2-sphere, $uv$ be an edge of $\cS$, $d\geq 4$ be an integer, and $Z$ be a set of $(d-3)$ vertices disjoint from $V(\cS)$. Suppose that $H_{d-2}(\cS/uv*Z)$ is volume rigid in $\R^d$. 
Then $H_{d-2}(\cS*Z)$ is volume rigid in $\R^{d}$. 
\end{lemma}
\begin{proof} We have $(\cS*Z)/uv=(\cS/uv)*Z$, and  so $H_{d-2}((\cS/uv)*Z)=H_{d-2}((\cS*Z)/uv)$. Since $H_{d-2}((\cS/uv)*Z)$ is volume rigid in $\R^{d}$, we can complete the proof by applying Lemma \ref{lem:vsplit_kface} to $H=H_{d-2}(\cS*Z)$. It will suffice to show that the matrix $A_{uv}(H,q,\bd)$ in the statement of Lemma \ref{lem:vsplit_kface} has rank $d$ for a generic realisation $q$ of $H/uv$ in $\R^d$ and a generic vector $\bd\in \R^d$. And since, by Lemma \ref{lem:rational}, the entries of $A_{uv}(H,q,\bd)$ are rational functions of the components of $q,\bd$ we can accomplish this by showing that $\rank A_{uv}(H,q,\bd)=d$ for some $q,\bd$, 
irrespective of whether $(H/uv,q)$ is infinitesimally volume rigid in $\mathbb R^d$ or not.

Let $Z=\{z_1,z_2,\ldots,z_{d-3}\}$ and let
$\Delta_x=\{u,v,x\}$ and $\Delta_y=\{u,v,y\}$ be the two 2-faces of $\cS$ which contain $uv$. 
Put $q(v)={\bf 0}$, $q(z_i)=\be_{i}$ for $1\leq i\leq d-3$, 
$q(x)=\be_{d-2}$, $q(y)=\be_{d-1}$  and choose $\bd=(a_1,a_2,\ldots,a_{d})\in \R^{d}$.  Let $\Delta_i=\Delta_x\cup (Z-z_i)$ for $1\leq i\leq d-3$, $\Delta_{d-2}= \Delta_y\cup (Z-z_{d-3})$, $\Delta_{d-1}=( \Delta_x-u)\cup Z$ and $\Delta_{d}=( \Delta_y-u)\cup Z$. Then, adopting the notation of Lemma \ref{lem:vsplit_kface}, we have $\Delta_i\in E_{uv}$ for $1\leq i\leq d-2$ and $\Delta_i\in E_{v}^u$ for $ i= d-1,d$.

Then, the row of $A=A_{uv}(H,q,\bd)$ indexed by $\Delta_i$, $1\leq i\leq d-2$, is  given by the projection of $\bd$ onto  $\langle q(w)-q(v):w\in \Delta_i-u\rangle^\perp$, which is equal to $a_i\be_i+a_{d-1}\be_{d-1}+a_{d}\be_{d}$ when $1\leq i\leq d-3$, and is equal to  $a_{d-3}\be_{d-3}+a_{d-2}\be_{d-2}+a_{d}\be_{d}$ when $i = d-2$.
Similarly, the rows of $A$ indexed by $\Delta_{i}$  are given by $q(v)^{\Delta_i}$ for $i=d-1,d$, i.e.~the projection of $q(v)=\bf 0$  onto the affine span of $q(\Delta_{i}-v)$,  
which is equal to $\frac{1}{d-2}(\be_{1}+\be_2+\ldots+\be_{d-2})$ when $i=d-1$ and $\frac{1}{d-2}(\be_{1}+\be_2+\ldots+\be_{d-3}+\be_{d-1})$ when $i=d$. Hence, up to scalar multiplication of the last two rows, the submatrix of $A$ containing these $d$ rows is: 
$$
B = \begin{pmatrix}
a_1 & 0 &\ldots& 0 &0&0&   a_{d-1} & a_{d}\\ 
0 & a_2&\ldots& 0 &0&0&   a_{d-1} & a_{d}\\
\vdots & \vdots& \vdots &\vdots&\vdots&\vdots&   \vdots &\vdots\\
0 & 0&\ldots& 0 &a_{d-3} & 0 &   a_{d-1} & a_{d}\\
0 & 0&\ldots& 0 &a_{d-3} &a_{d-2} & 0   & a_{d}\\
1 & 1&\ldots& 1 &1 &1 & 0   & 0\\
1 & 1&\ldots& 1 &1 &0 & 1   & 0
\end{pmatrix}.
$$
It is straightforward to check when $a_1=a_2=\ldots=a_{d-4}=a_{d-2} = a_{d}=1$ and $a_{d-3} = a_{d-1} = 0$, the determinant of $B$ is $1$.
For example, using a block decomposition of this specialisation of $B$ we have 
$$\begin{array}{rl}    
\det(B) = \det\begin{pmatrix}
    I_{d-4}& Y \\ U & V 
\end{pmatrix} & = \det(I_{d-4}) \det(V - U I_{d-4}^{-1}Y) \\
& = \det\begin{pmatrix}
    0&0&0& 1 \\
    0 & 1 & 0 & 1 \\
    1 & 1 & 0 & 4-d \\
    1 & 0 & 1 & 4-d 
\end{pmatrix} = 1.
\end{array}
$$
Hence  $A$ has rank $d$ for all generic $q$ and ${\bd}$.
\end{proof}

\begin{lemma} \label{lem:plane-cone}
Let  $\cS$ be a simplicial 2-sphere and $Z$ be a set of $(d-3)$ vertices disjoint from $V(\cS)$ for some   $d\in \{4,5,6\}$. Suppose that   $|V(\cS)|\geq 5$ when  
$d\in \{5,6\}$.
Then $H_{d-2}(\cS*Z)$ is volume rigid in $\R^{d}$. 
\end{lemma}
\begin{proof} We proceed by induction on $|V(\cS)|$. The inductive step is given by Lemma \ref{lem:cone} and a classical result on simplicial  
polyhedra, see for example \cite{RS}, that  tells us that  there exists an edge $uv$ of $\cS$ such that $\cS/uv$ is a simplicial 2-sphere whenever $|V(\cS)|\geq 5$. Hence we need only verify the base cases of our induction.

 When $d=4$ and $|V(\cS)|=4$, $\cS$ is the boundary complex of the 3-simplex,
$H_{2}(\cS*z)=K_{5}^3$, and $H_{2}(\cS*z)$ is volume rigid in $\R^{4}$ by Lemma \ref{lem:simplex}. 

When  $d\in \{5,6\}$ and $|V(\cS)|=5$, $\cS$ is the boundary complex of the triangular bipyramid. We have verified that $H_{d-2}(\cS*Z)$ is volume rigid in $\R^{d}$ in this case for both $d=5,6$ by computer, using an exact computation over a finite field. In particular we have verified that the following realisation of $H_{d-2}(\cS*Z)$ in $\R^d$ is infinitesimally volume rigid for $d=5,6$. Let $v_1,v_2,v_3$ be the vertices of the base of the triangular bipyramid, $v_4,v_5$ be the apex vertices and $Z=\{z_1,z_2,\ldots,v_{d-3}\}$. Put $p(v_i)=\be_i$ for $i\in \{1,2,3\}$, $p(v_4)=\be_3-\be_1$, $p(v_5)=\be_1-\be_2$ and $p(z_i)=\be_{i+3}$ for $1\leq i\leq d-3$. 
\end{proof}

\begin{theorem} \label{thm:d=5}
Let $\cS$ be a connected, simplicial homology $(d-1)$-manifold for some $d\in \{4,5,6\}$. Then $H_{d-2}(\cS)$ is volume rigid in $\R^d$.
\end{theorem}
\begin{proof}
Suppose, for a contradiction, that $H_{d-2}(\cS)$ is not volume rigid in $\R^{d}$. 
Choose a $(d-4)$-face $Z$ of $\cS$.
Then $\lk(Z)$ is a simplicial 2-sphere
and $\st(Z)$ is the  simplicial cone of $\lk(Z)$ with $Z$, so $H_{d-2}(\st(Z))$ is volume rigid in $\R^{d}$  by Lemma \ref{lem:plane-cone}, as long as $|V(\lk(Z))|\geq 5$  when $d\in \{5,6\}$. 

Consider the following two cases.  
\\[1mm]
{\bf \boldmath Case 1: 
there exists $H\subseteq H_{d-2}(\cS)$ and a $(d-1)$-face $X$ of $\cS$ such that $X\subseteq V(H)$ and $H$ is volume rigid in $\R^{d}$.}
We may suppose that $H$ has been chosen to be a maximal such hypergraph.  Since $\cS$ is a connected, simplicial homology $(d-1)$-manifold, it is strongly connected. Since $V(H)\neq V(\cS)$, otherwise $H_{d-2}(\cS)$ would be volume rigid in $\R^{d}$, this implies that we can find two $(d-1)$-faces $X_1,X_2$ of $\cS$ with $X_1\subseteq  V(H)$, $X_2\not\subseteq V(H)$ and $|X_1\cap X_2|= d-1$. Choose $Z\subseteq X_1\cap X_2$ with $|Z|=d-3$. Then $X_1,X_2\in \st(Z)$. If $d=4$ or $|V(\lk(Z))|\geq 5$ then $H_{d-2}(\st(Z))$ is volume rigid in $\R^{d}$ by the first paragraph of the proof. This would imply that  $H\cup H_{d-2}( \st(Z))$ is volume rigid in $\R^{d}$ by Lemma \ref{lem:glue_k}, and contradict the maximality of $H$. Hence $d\in \{5,6\}$, $|V(\lk(Z))|= 4$ and $\lk(Z)$ is the boundary complex of a 3-simplex. 
 
Let $X_1\cap X_2=Z'$, $X_1\sm X_2=\{x_1\}$ and $X_2\sm X_1=\{x_2\}$. Then 
$V(\lk(Z))=\{x_1,x_2\}\cup (Z'\sm Z)$. Since $\lk(Z)$ is the boundary complex of a 3-simplex, we have $\{x_1,x_2,z\}$ is a 2-face of $\lk(Z)$ for all $z\in Z'\sm Z$. Hence $\{x_1,x_2,z\}\cup Z''$ is a $(d-2)$-face of $\cS$ for all $z\in Z'\sm Z$ and $Z''\subseteq Z$ with $|Z''|=d-4$.  Since $Z$ is an arbitrary subset of $X_1\cap X_2$ of size $d-3$, this implies that $\{x_1,x_2\}\cup Z''$ is a $(d-2)$-face of $\cS$ for all $Z''\subseteq Z'$ with $|Z''|=d-3$. Since $X_i\in \cS$ for $1\leq i\leq 2$, we also have $Z''$ is a $(d-2)$-face of $\cS$ for all $Z''\subseteq X_1\cup X_2$ such that  $|Z''|=d-1$ and  $\{x_1,x_2\}\not\subset Z''$.
This implies that $H'=H_{d-2}(\cS)[X_1\cup X_2]\cong K_{d+1}^{d-1}$ and hence $H'$ is volume rigid in $\R^d$ by Lemma \ref{lem:simplex}. Since $X_1$ is a hyperedge of both $H$ and $H'$,  we can now apply Lemma  \ref{lem:glue_k} to deduce that $H\cup H'$ is volume rigid in $\R^d$ and contradict the maximality of $H$. 
 \\[1mm]
{\bf \boldmath Case 2:  $H$ is not volume rigid in $\R^{d}$ for all $H\subseteq H_{d-2}(\cS)$ such that $V(H)$ contains a $(d-1)$-face of $\cS$.} We proceed as in Case 1. Since $\cS$ is a connected, homology simplicial $(d-1)$-manifold, it is strongly connected and hence  there exist two $(d-1)$-faces $X_1,X_2$ of $\cS$ with  $|X_1\cap X_2|= d-1$. Choose $Z\subseteq X_1\cap X_2$ with $|Z|=d-3$. Then $X_1,X_2\in \st(Z)$. If $d=4$ or $|V(\lk(Z))|\geq 5$ then $H_{d-2}(\st(Z))$ is volume rigid   in $\R^{d}$ by the first paragraph of the proof. This would 
contradict the hypothesis of Case 2 since $X_1\in \st(Z)$. Hence $d\in \{5,6\}$, $|V(\lk(Z))|= 4$ and $\lk(Z)$ is the boundary complex of a 3-simplex. We can now apply the argument of Case 1 to deduce that $H'=H_{d-2}(\cS)[X_1\cup X_2]\cong K_{d+1}^{d-1}$ and hence $H'$ is volume rigid in $\R^d$ by Lemma \ref{lem:simplex}. This again contradicts the hypothesis of Case 2 and completes the proof of the theorem.
\end{proof}

In order to extend the above proof technique to verify Conjecture \ref{con:d-2} for all $d\geq 3$, we will need an analogue of Lemma \ref{lem:plane-cone}
for all $d\geq 4$. The inductive proof of this lemma would extend to all $d\geq 4$  if we could verify the following conjecture as the base case of our induction. 

\begin{conjecture} \label{lem:plane d-cone}
Let  $d\geq 4$ be an integer and $\cS$ be a simplicial 2-sphere with  $\lceil d/2\rceil+2$ vertices. Let $\cS*Z$ be the simplicial cone of $\cS$ with a set $Z$ of $d-3$ new vertices.
Then $H_{d-2}(\cS*Z)$ is volume rigid in $\R^{d}$. 
\end{conjecture}

The hypothesis that $|V(\cS)|\geq \lceil d/2\rceil+2$ is needed in Conjecture \ref{lem:plane d-cone} since, if $|V(\cS)|\leq\lceil d/2\rceil+1$,  then a straightforward calculation shows that $|E(H_{d-2}(S*Z))| < d|V(S*Z)|-\binom{d+1}2$, so $S*Z$ cannot be volume rigid in $\mathbb R^d$. Note, however, that even if we could verify Conjecture \ref{lem:plane d-cone}, we would still need to extend the argument in Cases 1 and 2 of the proof of Theorem \ref{thm:d=5} to cover the situation when $5\leq |V(\lk(Z))|\leq\lceil d/2\rceil+1$ in order to verify Conjecture \ref{con:d-2}.

\section{Closing remarks and open problems}

\subsection{Volume rigidity of $d$-uniform hypergraphs in $\R^d$}
It seems to be difficult to find families of  $d$-uniform hypergraphs which are volume rigid in $\R^d$.
One notable exception is the following recent strengthening of Theorem \ref{thm:complete} by  Lew, Nevo, Peled, and Raz \cite[Theorem 2]{LNPR} to include complete $d$-uniform hypergraphs with at least $d+2$ vertices.

\begin{theorem}\label{thm:LNPR}
$K_n^{k+1}$ is volume rigid in $\R^d$ if $1\leq k\leq d-2$ and $n\geq d+1$, or if $1 \leq k = d-1$ and $n \geq d+2$.    
\end{theorem}
\noindent {We became aware of \cite{LNPR} shortly before submitting this paper. The proof techniques of Theorem \ref{thm:complete} and \cite[Theorem 2]{LNPR} are significantly different.  Both results were obtained independently and contemporaneously.
}

Note that the  $k=d$ case of Theorem~\ref{thm:LNPR} does not extend  to $(d-1)$-skeleton hypergraphs of simplicial $(d-1)$-manifolds since these hypergraphs may not have enough hyperedges to be volume rigid in $\R^d$. Consider, for example, the $2$-skeleton hypergraph of   a simplicial 2-sphere in $\R^3$.

\subsection{Infinitesimal volume rigidity of convex simplicial polytopes}
Whiteley \cite{W84} showed that, when $d\geq 3$, the 1-skeleton of every convex simplicial $d$-polytope is infinitesimally rigid in $\R^d$ (without any assumption of genericity). Could this result be extended to show that the $k$-skeleton of every convex simplicial $d$-polytope is infinitesimally volume rigid in $\R^d$  for all  $1\leq k\leq d-2$?

\section*{Acknowledgements}
We would like to thank David Ellis for his suggestion that we could use the result from \cite{L} to simplify our proof of Lemma \ref{lem:simplex}, and Alan Lew, Yuval Peled, Eran Nevo and Orit Raz for sharing a draft of their preprint \cite{LNPR} just before both articles were submitted to arXiv. We also thank the anonymous referees for many helpful suggestions that improved the paper, including the statement of Theorem \ref{thm:new}.
This work was supported by JST PRESTO Grant Number JPMJPR2126, JSPS KAKENHI Grant Number 20H05961, and EPSRC overseas travel grant EP/T030461/1.

\end{document}